\documentclass{article}
\usepackage{graphicx} % Required for inserting images
\usepackage{amsmath,amssymb,amsthm}
\usepackage{dsfont}
\usepackage{color}
\usepackage{lineno,hyperref}
\usepackage{yfonts}
\usepackage{appendix}
\usepackage{mathtools}
\usepackage{bm}
\usepackage{enumerate}
\usepackage[usenames,dvipsnames]{xcolor}
\usepackage[mathscr]{eucal}

\DeclareMathOperator*{\divv}{div}

\newcommand*{\Bscr}{\mathcal B}
\newcommand*{\Cscr}{\mathcal C}

\newcommand*{\Fscr}{\mathcal F}

\newcommand*{\Pscr}{\mathcal P}

\newcommand*{\N}{\mathbb{N}}
\newcommand*{\E}{\mathbb{E}}
\newcommand*{\R}{\mathbb{R}}

\numberwithin{equation}{section}
\newtheorem{theorem}{Theorem}[section]
\newtheorem{lemma}[theorem]{Lemma}

\newtheorem{proposition}[theorem]{Proposition}
\newtheorem{claim}{Claim}

\newtheorem*{definition*}{Definition}
\newtheorem*{proposition*}{Proposition}

\theoremstyle{definition}
\newtheorem{definition}[theorem]{Definition}

\theoremstyle{definition}
\newtheorem{remark}[theorem]{Remark}
\newtheorem{example}[theorem]{Example}

\usepackage{xparse}

\ExplSyntaxOn
\NewDocumentCommand{\makeabbrev}{mmm}
 {
  \yoruk_makeabbrev:nnn { #1 } { #2 } { #3 }
 }

\cs_new_protected:Npn \yoruk_makeabbrev:nnn #1 #2 #3
 {
  \clist_map_inline:nn { #3 }
   {
    \cs_new_protected:cpn { #2 } { #1 { ##1 } }
   }
 }
 \ExplSyntaxOff
 
\makeabbrev{\textbf}{tbf#1}{a,b,c,d,e,f,g,h,i,j,k,l,m,n,o,p,q,r,s,t,u,v,w,x,y,z,A,B,C,D,E,F,G,H,I,J,K,L,M,N,O,P,Q,R,S,T,U,V,W,X,Y,Z}

\makeabbrev{\textbf}{bf#1}{a,b,c,d,e,f,g,h,i,j,k,l,m,n,o,p,q,r,s,t,u,v,w,x,y,z,A,B,C,D,E,F,G,H,I,J,K,L,M,N,O,P,Q,R,S,T,U,V,W,X,Y,Z}

\makeabbrev{\textsf}{tsf#1}{a,b,c,d,e,f,g,h,i,j,k,l,m,n,o,p,q,r,s,t,u,v,w,x,y,z,A,B,C,D,E,F,G,H,I,J,K,L,M,N,O,P,Q,R,S,T,U,V,W,X,Y,Z}

\makeabbrev{\mathsf}{mss#1}{a,b,c,d,e,f,g,h,i,j,k,l,m,n,o,p,q,r,s,t,u,v,w,x,y,z,A,B,C,D,E,F,G,H,I,J,K,L,M,N,O,P,Q,R,S,T,U,V,W,X,Y,Z}

\makeabbrev{\mathfrak}{mf#1}{a,b,c,d,e,f,g,h,i,j,k,l,m,n,o,p,q,r,s,t,u,v,w,x,y,z,A,B,C,D,E,F,G,H,I,J,K,L,M,N,O,P,Q,R,S,T,U,V,W,X,Y,Z}

\makeabbrev{\mathrm}{mrm#1}{a,b,c,d,e,f,g,h,i,j,k,l,m,n,o,p,q,r,s,t,u,v,w,x,y,z,A,B,C,D,E,F,G,H,I,J,K,L,M,N,O,P,Q,R,S,T,U,V,W,X,Y,Z}

\makeabbrev{\mathbf}{mbf#1}{a,b,c,d,e,f,g,h,i,j,k,l,m,n,o,p,q,r,s,t,u,v,w,x,y,z,A,B,C,D,E,F,G,H,I,J,K,L,M,N,O,P,Q,R,S,T,U,V,W,X,Y,Z}

\makeabbrev{\mathcal}{mc#1}{A,B,C,D,E,F,G,H,I,J,K,L,M,N,O,P,Q,R,S,T,U,V,W,X,Y,Z}

\makeabbrev{\mathbb}{mbb#1}{A,B,C,D,E,F,G,H,I,J,K,L,M,N,O,P,Q,R,S,T,U,V,W,X,Y,Z}

\makeabbrev{\mathscr}{ms#1}{A,B,C,D,E,F,G,H,I,J,K,L,M,N,O,P,Q,R,S,T,U,V,W,X,Y,Z}

\makeabbrev{\mathrm}{#1}{
Id,id,ran,rk,diag,stab,ann,conv,pr,ev,tr,End,Hom,sgn,im,op,can,fin,ext,%red,tot,
rot,usc,lsc,Lip,LocLip,lip,bSymLip,osc,AC,loc,uloc,spec,coz,z,ul,
supp,Opt,Adm,Cpl,Geo,GeoSel,GeoOpt,GeoAdm,GeoCpl,reg,
bd,co,Ric,Exp,dExp,dist,seg,Seg,cut,fcut,Cut,SDiff,Iso,Isom,diam,cl,Homeo,Diff,Der,vol,dvol,inj,relint, Graph, sub,codim,
var,law,Var,Poi,Gam,pa,so,iso,fs,inv,pqi,mix,
TestF,
}

\newcommand{\Pbb}{\mathbb P}
\newcommand{\dd}{\mathrm d}
\newcommand{\diver}{\operatorname{div}}
\title{Brownian motion in Minkowski normed spaces}
\author{%
    Shin-ichi Ohta\footnote{Department of Mathematics, University of Osaka, Osaka 560-0043, Japan \& RIKEN Center for Advanced Intelligence Project (AIP), 1-4-1 Nihonbashi, Tokyo 103-0027, Japan. E-mail: s.ohta@math.sci.osaka-u.ac.jp}
    \and Marco Rehmeier\footnote{Institute of Mathematics, TU Berlin, 10623 Berlin, Germany. E-mail: rehmeier@tu-berlin.de} 
    \and Kohei Suzuki\footnote{Department of Mathematical Sciences, Durham University, South Road, Durham
DH1 3LE, United Kingdom. E-mail: kohei.suzuki@durham.ac.uk} 
}
\date{}
\begin{document}
	\maketitle

	\begin{abstract}
    A Minkowski normed space is the Euclidean space equipped with a (possibly asymmetric) uniformly convex and smooth norm, forming a particular class of Finsler manifolds. 
    We construct a stochastic process with one-dimensional time marginal densities given by the fundamental solution to the nonlinear Finsler heat equation in Minkowski normed spaces.
    This process is constructed as a solution to a singular McKean--Vlasov stochastic differential equation and constitutes a nonlinear Markov process in the sense of McKean. 
    We show that solutions to this stochastic differential equation are pathwise unique, and, thus, probabilistically strong solutions, though the equation has singular coefficients beyond the subcritical regime. 
    Since our construction is a natural extension of the construction of standard Brownian motion from the standard heat kernel, we call this process \emph{Brownian motion in Minkowski normed spaces.} 
    Furthermore, we obtain a universality property, namely that whichever Minkowski norm is chosen,  the radial part of this process is the standard Bessel process.
    To the best of our knowledge, this is the first construction of stochastic processes associated with nonlinear heat equation in Finslerian spaces.
	\end{abstract}
    

    \bigskip

%\cyan{If we do not assume the eveness, it is not anymore a normed space. The title should be changed.}
%\blue{SO: No, Minkowski norm includes such asymmetric cases.}

\section{Introduction} 
\begingroup
\renewcommand{\thetheorem}{\arabic{theorem}}
\subsection{Framework and goal}
One of the cornerstones of stochastic analysis and linear PDE theory is the intimate relation between heat equation and Brownian motion, namely that the fundamental solutions to 
\begin{equation}\label{intro:HE}
    \partial_t u = \frac 1 2 \Delta u
\end{equation}
(i.e., the classical heat kernel $p(t,z,x)$) are the one-dimensional time marginal densities of Brownian motion; more precisely, the density of the distribution of $W^z_t:= W_t +z$ is $p(t,z,\cdot)$, where $W$ denotes standard Brownian motion. 
Moreover, the path laws $\{P^z\}_{z\in \R^d}$ of $W^z$ form a Markov process, which is uniquely determined by its transition kernel $p(t,z,x)$. These links, allowing to pose analytic problems for \eqref{intro:HE} as probabilistic ones and vice versa, is fundamental, for instance in classical potential theory (see, among many other references, \cite{BG68,D84,FOT11}), and
inspired a powerful general theory of linear PDEs, stochastic processes and Markov processes. One central aspect of this theory is the construction of probabilistic counterparts for solutions $u$ to large classes of PDEs of type
\[
\partial_t u = \Delta \bigl( a(x) u \bigr) - \nabla \cdot \bigl( b(x) u \bigr)
\]
for $a:\R^d \to [0,\infty)$ and $b: \R^d \to \R^d$. 
Typically, these counterparts are constructed as solutions $X$ to the associated stochastic differential equation 
\[
\dd X_t = b(X_t) \,\dd t + \sqrt{2a(X_t)} \,\dd W_t
\]
such that the distribution of $X_t$ is $u(t,x) \,\dd x$ for all $t\geq 0$ (for the heat equation, one has $a = 1/2$ and $b= 0$).

\paragraph{Heat equation in Minkowski normed spaces.}
A \emph{Minkowski normed space} $(\R^d,\|\cdot\|)$ is the Euclidean space $\R^d$ equipped with a uniformly convex,  smooth and possibly asymmetric norm (i.e., $\|{-x}\| \neq \|x\|$ in general); see Section \ref{sect.2.1}. Since such a norm does not necessarily arise from an inner product, the associated geometry is typically \textit{Finslerian} rather than \textit{Riemannian}, where the infinitesimal structures of the space (i.e., tangent spaces) are not equipped with a scalar product, but with a (possibly asymmetric) norm.
% The aim of this paper is to explore the possibility of extending this deep relation to nonlinear heat equation.
% Precisely, we consider a \emph{Minkowski normed space} $(\R^d,\|\cdot\|)$, by which we mean a Euclidean space $\R^d$ equipped with a uniformly convex and smooth, possibly asymmetric norm $\|\cdot\|$
% (we also assume $d \ge 2$ for a technical reason; see ...).
% This is the simplest example of a (non-Riemannian) Finsler manifold.
In the context of Finsler manifolds, the natural \emph{nonlinear Laplacian} $\Delta$  can be defined analogously to the Riemannian case.
%; see, for instance, %\cite{GeShen,OhtaSturm2009}. 
In the present Minkowski setting, it is explicitly given as
\begin{equation}\label{eq:FLap}
    \Delta u=\sum_{i,j=1}^d \partial_i \bigl( \partial_{ij} Q^*(\nabla u) \partial_j u \bigr),
\end{equation}
where $Q^*$ is the Legendre dual of the Lagrangian $Q:=\|\cdot\|^2/2$ and $\nabla u:=(\partial_1 u,\ldots,\partial_d u)$ ($\partial_iu$ denotes the standard partial derivative of $u$). 
When $\|\cdot\| = |\cdot|$ is the standard Euclidean norm, \eqref{eq:FLap} is the classical Laplacian. 
The notable feature is that the operator in \eqref{eq:FLap} is generally {\it nonlinear} ($\Delta(u_1+u_2) \neq \Delta u_1 + \Delta u_2$) and {\it inhomogeneous} ($\Delta(-u) \neq -\Delta u$) due to the non-constancy of $\partial_{ij}Q^*$ and the asymmetry of $Q$. 
Analogously to the classical heat equation, the {\it Finsler heat equation} is 
\begin{equation}\label{eq:FHE}
    \partial_t u = \frac{1}{2} \Delta u.
\end{equation}
In the Minkowski case, the \textit{Finsler heat kernel} centered at $z \in \R^d$ is given by 
\begin{equation}\label{eq:FHK}
f_Q^z(t,x):=\frac{1}{(2t)^{d/2}C_Q}
\exp\biggl(-\frac{Q(z-x)}{t}\biggr),
\end{equation}
with the normalizing constant~$C_Q>0$ (see \eqref{eq:def-kernel} for the precise definition of $C_Q)$. This function solves \eqref{eq:FHE} and is also called the \emph{fundamental solution} to \eqref{eq:FHE} (see \cite[Example 4.3]{OhtaSturm2009}).  
% Because of its similarity to the usual heat kernel on $\R^d$, we will call $f^z_Q$ the \emph{Finsler heat kernel}.

Heat equation on Finslerian spaces has attracted great attention across nonlinear PDEs, geometric analysis and optimal transport. 
From the PDE point of view, although $\Delta$ is nonlinear, it is uniformly elliptic and the existing PDE techniques apply to some extent; we refer to \cite{GeShen,Ohta2017,Ohta2021,Ohta2022,OhtaSturm2009,OhtaSturm2014,OhtaSuzuki2025} for some results. 
However, the noncompact case has certain limitations: some basic geometric and analytic inequalities are known only under additional assumptions; see, e.g., \cite{Ohta2022}.
From the optimal transport point of view, the heat equation~\eqref{eq:FHE} can be regarded as the gradient flow equation for the Dirichlet energy in the $L^2$-space or for the relative entropy in the (reverse) $L^2$-Wasserstein space (see \cite{OhtaSturm2009,OhtaZhao2025}). 

A fundamental missing piece, however, lies on the probabilistic side: unlike the linear heat equation on Riemannian manifolds, where the heat equation is inseparably linked to Brownian motion and stochastic differential equations, the nonlinear Finsler heat equation has so far lacked a comparable stochastic counterpart.

%On the other hand, there are also essential differences from the linear case.
%one of the most crucial difficulties is the lack of contraction (non-expansion) property with respect to the Wasserstein distance (see \cite{OhtaSturm2012}).
%Due to the absence of contraction property, gradient flow theory cannot ensure that $\mu_t:=f^z_Q(t,x) \,\dd x$ extended to $t=0$ as $\mu_0=\delta_z$ is a unique solution to \eqref{eq:FHE} stemming from $\delta_z$.
%Among others, for noncompact Finsler manifolds, some fundamental geometric and analytic inequalities are known only under additional conditions (see, e.g., \cite{Ohta2022}).
%(\red{M: I dont understand the mathematical content of the previous sentence}).
%This prevents further \red{ad hoc} probabilistic studies of Finsler heat flow.

\paragraph{Goal.}An open question is whether there is a canonical notion of \emph{Brownian motion} in this setting. More fundamentally, one may ask what the appropriate probabilistic counterpart of the nonlinear Finsler Laplacian should be. 

A na\"ive analogy with the classical Riemannian case would suggest defining Brownian motion as the Markov process whose transition semigroup solves the Finsler heat equation. However, this analogy immediately breaks down at a fundamental level: the Finsler heat equation is nonlinear because  the Finsler Laplacian itself is nonlinear while the transition semigroup of any Markov process is necessarily linear. Thus, a classical Markov process cannot have the nonlinear Finsler heat flow as its transition semigroup.
Even if one leaves aside Markovianity, it is an open problem how to construct a stochastic process with one-dimensional time marginal densities given by the fundamental solution to the Finsler heat equation.

The main goal of this paper is to provide an answer to this problem for
Minkowski normed spaces by constructing a unique nonlinear Markov process
whose one-dimensional time marginal densities are given by the fundamental
solution to the Finsler heat equation \eqref{eq:FHE}, i.e., by \eqref{eq:FHK}. Our approach
extends a classical construction of Brownian motion via the classical
heat equation to the Minkowski setting. Therefore, we call this process \emph{Brownian motion in Minkowski normed spaces}. This is, to the best of our knowledge, the first construction of Brownian motion in the Finslerian
setting.

Moreover, we prove a universality property of this process, namely that its radial part equals the standard $d$-dimensional Bessel process, independently of the Minkowski norm under consideration. We believe that this universality can be of independent interest and helpful for further analysis of this process.

The notion of \textit{nonlinear Markov process} used in this paper (see Definition \ref{def:NL-MP}) is based on ideas by McKean \cite{McKean1966} and was recently introduced and studied in \cite{RR25}.
In contrast to classical Markov processes, the transition law may depend not only on the current state, but also on the current distribution of the process.

\subsection{Main results}

\paragraph{Setting.}
In this paper, we work with a Minkowski norm $\|\cdot\|$ on $\R^d$ with $d \ge 2$, being uniformly convex and uniformly smooth, but possibly asymmetric. 
More precisely, in terms of the Lagrangian $Q:=\|\cdot\|^2/2:\R^d \to [0,\infty)$, we impose the following conditions: 
\begin{enumerate}
\item[(H1)]
$Q$ is positively $2$-homogeneous:
\[
Q(cx)=c^2 Q(x), \qquad \forall c>0,\ x\in\R^d;
\]
\item[(H2)] $Q\in C^1(\R^d)\cap C^3(\R^d\setminus\{0\})$; 
\item[(H3)] $\exists 0<\lambda\le \Lambda<\infty$ such that for every $x \in \R^d \setminus \{0\}$, as quadratic forms,
\[
\lambda I\le D^2Q(x)\le \Lambda I.
\]
\end{enumerate}

\paragraph{The McKean-Vlasov SDE and nonlinear Markov property.}
As our main result, we construct a nonlinear Markov process whose one-dimensional time-marginal densities are given by the Finsler heat kernel $\{f_{Q}^z(t,\cdot)\}_{z\in \R^d, t>0}$ as in \eqref{eq:FHK}, i.e., by the fundamental solution to the Finsler heat equation~\eqref{eq:FHE}.
To this end, we first associate the Finsler heat equation to a distribution-dependent McKean--Vlasov SDE, and then prove the existence and uniqueness of  probabilistically weak solutions to this SDE. We also show that these solutions are, in fact, probabilistically strong. The coefficients of this equation are given as follows:
\begin{equation*}%\label{coeffs-A,B}
    \begin{aligned}
    a_{ij}(x,\varphi) &:= \frac 1 2\partial_{ij} Q^*\bigl(\nabla \varphi (x)\bigr), \\
    b_i(x,\varphi) &:= \partial_j a_{ij}(x,\varphi) = \frac 1 2 \partial_j \Bigl( \partial_{ij}Q^*\bigl(\nabla \varphi(x)\bigr)\Bigr),\quad (x,\varphi) \in \mathbb{R}^d \times \mathcal{C}
    \end{aligned}
\end{equation*}
(see \eqref{coeffs-A,B} for details, in particular for the domain $\mathcal{C}$ of $a_{ij}$ and $b_i$). 
In the following theorem, we summarize the main results of this paper. We denote by $\mathcal{L}(X)$ the distribution of a random variable $X$ ($\mathbb{R}^d$- or path-valued; in the latter case, $\mathcal{L}(X)$ is the path law of a stochastic process $X$).

\begin{theorem}[see Theorems \ref{thm:construction}, \ref{thm:main-result} and Proposition~\ref{prop:strong-sol} for details] \label{thm1-intro}
Let $d\geq 2$ and assume \textnormal{(H1)--(H3)}. 
For each $z\in \R^d$, the McKean--Vlasov SDE
\begin{equation}\label{intro:DDSDE}
\left\{
\begin{aligned}
    \,\dd X_t &= b\bigl( X_t,u(t) \bigr) \,\dd t + \sqrt{2a \bigl( X_t,u(t) \bigr)} \,\dd W_t,
    \\ \mathcal{L}(X_t) &= u(t,x) \,\dd x,\ \forall t>0,\quad X_0 = z
\end{aligned}
\right.
\end{equation}
 has a unique probabilistically weak solution $X^z$ with the prescribed one-dimensional time marginals $\mathcal{L}(X^z_t) = f^z_Q(t,x) \,\dd x$ for all $t>0$. Furthermore, the path laws $P^z := \mathcal{L}(X^z)$ of $X^z$ form a nonlinear Markov process with one-dimensional time marginal densities $f^z_Q(t,\cdot)$, $(t,z)\in (0,\infty) \times \R^d$. Finally, the solutions $X^z$ are probabilistically strong after every strictly positive time.
\end{theorem}
This construction generalizes the construction of standard Brownian motion as the unique (classical) Markov process with one-dimensional time marginal densities given by the classical heat kernel.
Indeed, for the special case $Q(\cdot) = |\cdot|^2/2$, where $|\cdot|$ denotes the usual Euclidean norm, \eqref{intro:DDSDE} becomes $\dd X_t =\dd W_t$, and $X^z$ from Theorem \ref{thm1-intro} becomes $X^z = W +z$ for the standard Brownian motion $W$.
%This close analogy suggests the following definition.
Motivated by this analogy, we propose the following definition.
\begin{definition*}[see Definition \ref{def:FBm}]\label{def:intro-fbm}
    We call $X^z$ from Theorem \ref{thm:main-result} \emph{Brownian motion starting at $z$ in the Minkowski normed space \((\R^d, \|\cdot\|)\)}.
\end{definition*}

As another main result, we prove the following universality result:
\begin{theorem}[see Theorem \ref{thm:universality} for details]\label{thm:universality-intro}
    Let $d\geq 2$. Then, for $X^z$ as above, $(\|z-X^z_t\|)_{t\geq 0}$ is the standard $d$-dimensional Bessel process. 
\end{theorem}

% \cyan{KS.When $d=1$, the norm is trivialised (i.e., becomes the standard Euclidean norm), so it is anyway not our main interest. We can simply remove the following paragraph?}
% \blue{SO: Removing the evenness gives a nonstandard structure even when $d=1$.}
% \cyan{KS. I would suggest not mentioning $d=1$ because I am not sure if Lemma 2.7 is only the part we need to modify, and also I am not aware of interesting examples with $d=1$ If we assume symmetry of the norm, it is the Euclidean norm. So, I am not sure if it is really a good idea to spare several lines why we cannot apply our proof for $d=1$.}
% To also cover the case $d=1$, one needs to modify the proof of integrability for $b$ within the proof of Lemma \ref{lem:f-solves-NLFPE}, but we refrain from doing so here. 
% That $a$ and $b$ depend on the distribution of the solution is a reflection of the nonlinearity of the Finsler heat equation. 

\paragraph{Outline of the proof of Theorem \ref{thm1-intro}.}
The proof of weak well-posedness of \eqref{intro:DDSDE} splits into an existence and a uniqueness part:
\smallskip

{\it Existence.}
The proof proceeds in three steps: first, we recast the Finsler heat equation as the \emph{nonlinear Fokker--Planck equation}
\begin{equation}\label{intr:nlFPE}
    \partial_t u(t,x) = \partial_{ij}\Bigl( a_{ij}\bigl( x,u(t) \bigr) u(t,x) \Bigr) - \partial_i \Bigl( b_i \bigl( x,u(t) \bigr) u(t,x) \Bigr);
\end{equation}
secondly, we prove that $f_Q^z$ is a distributional solution to~\eqref{intr:nlFPE} (Lemma \ref{lem:f-solves-NLFPE}); 
thirdly, we construct a weak solution to~\eqref{intro:DDSDE} by aid of the Ambrosio--Figalli--Trevisan \emph{superposition principle}, which allows to construct a solution to~the McKean--Vlasov SDE~\eqref{intro:DDSDE} with one-dimensional time marginal densities given by solutions to~\eqref{intr:nlFPE} (see Proposition \ref{prop:SP} below, \cite{Trevisan16} and also \cite{BR18,BR18_2}).

\smallskip

% {\it Uniqueness.} As our second main result, we employ and further develop results on the existence of nonlinear Markov processes with prescribed one-dimensional time marginals initiated in \cite{RR25} in order to prove that the solutions from Theorem \ref{thm1-intro} form a nonlinear Markov process. A key ingredient (and a result of independent interest) is the uniqueness of the solutions constructed in Theorem \ref{thm1-intro}.

{\it Uniqueness.} To prove the uniqueness, it suffices to show the uniqueness for the \emph{linearized Fokker--Planck equation} 
\begin{equation}\label{eqI:lin-FPE}
     \partial_t u(t,x) = \partial_{ij}\Bigl(a_{ij}\bigl(x,f^z_Q(s+t)\bigr) u(t,x) \Bigr)   -\partial_i \Bigl(b_i\bigl(x,f^z_Q(s+t)\bigr) u(t,x) \Bigr)
\end{equation}
for all $s>0$ and $z\in \R^d$ (see Proposition~\ref{prop:NL-MC-construction} and Lemma~\ref{lem:equiv-extrem-uniqu} for details).

To this end, we prove the strictly stronger result of pathwise uniqueness for the associated (non-distribution dependent) SDE with prescribed one-dimensional time marginals $f^z_Q(s+t,x) \,\dd x$, i.e.,
% \begin{equation}\label{I:SDE-linear}
%     \dd X_t = \frac 1 2 \divv \Bigl(D^2 Q^*\bigl(\nabla Q(z-X_t)\bigr)\Bigr) \,\dd t + \sqrt{D^2Q^*\bigl(\nabla Q(z-X_t)\bigr)} \,\dd W_t.
% \end{equation} 
\begin{equation} \label{I:SDE-linear}
\left\{
\begin{aligned}
    \dd X_t 
    &= b\bigl(X_t,f^z_Q(s+t)\bigr) \,\dd t + \sqrt{2a \bigl(X_t,f^z_Q(s+t)\bigr)} \,\dd W_t
   % \\ &=\frac 1 2 \divv \Bigl(D^2Q^*\bigl(\nabla f^z_Q(s+t,X_t)\bigr)\Bigr) \,\dd t + \sqrt{D^2Q^*\bigl(\nabla f^z_Q(s+t,X_t)\bigr)} \,\dd W_t
    \\&= \frac 1 2 \divv\Bigl(D^2Q^*\bigl(\nabla Q(z-X_t)\bigr)\Bigr) \,\dd t + \sqrt{ D^2Q^*\bigl(\nabla Q(z-X_t)\bigr)} \,\dd W_t,
    \\ \mathcal{L}(X_t) &= f^z_Q(s+t,x) \,\dd x,\ \forall t\geq0
\end{aligned}
\right.
\end{equation}
(see Lemma \ref{lem:nice-form} for the second equality).
% Note that, while \eqref{I:SDE-linear} is a classical SDE in the sense that the coefficients do not depend on the distributions, the coefficients of \eqref{I:SDE-linear} are generally singular in the sense that they do not belong to the sub-critical regime (i.e., the singular drift is not sufficiently controlled by the regularizing effect of Brownian noise under the natural parabolic scaling). Indeed, we can explicitly choose a Minkowski norm~\(Q\) that induces  critical or super-critical coefficients (\cyan{see, Example~??}). Therefore, the SDE does not fall within the well-established singular SDE framework of Krylov--R\"ockner~\cite{KrylovRoeckner2005}, nor, to the best of our knowledge, within any existing framework that would directly yield pathwise uniqueness. 

We point out that the SDE~\eqref{I:SDE-linear} does not fall within the
standard subcritical framework for the pathwise uniqueness of singular SDEs in Krylov--R\"ockner~\cite{KrylovRoeckner2005} nor within further developments~\cite{Grube2025,KinzebulatovMadou2024,Krylov2023Morrey,LingRoecknerZhu2019,vonderLuehe2018}. 
Indeed, our drift part can be as singular as 
\[
\Bigl| \divv\Bigl(D^2Q^* \bigl(\nabla Q(z-x) \bigr)\Bigr) \Bigr| \sim \frac{1}{|z-x|}
\qquad\text{as }x\to z
\]
(see Remark \ref{r:BS}, Example~\ref{ex:Randers}).
Thus the drift is locally in \(L^p\) only for \(p<d\). In particular, it lies outside the usual subcritical regime
\[
\frac{d}{p}+\frac{2}{q}<1,
\]
which requires \(p>d\). 
Moreover, the diffusion coefficient is also singular: its gradient can be as singular as 
\[
\left|\nabla \sqrt{D^2Q^*\bigl(\nabla Q(z-x) \bigr)}\right|\sim \frac{1}{|z-x|} \qquad \text{as }x\to z.
\]

Our pathwise uniqueness result for \eqref{I:SDE-linear} follows from the following result, which also leads to a straightforward proof of our second main result, Theorem~\ref{thm:universality-intro}.
\begin{proposition*}[Proposition~\ref{lem:no-hit-sde}] \label{lemI:no-hit-sde}
Let $s> 0$ and $X$ be a weak solution of~\eqref{I:SDE-linear} with initial datum $f_Q^z(s,x)\,\dd x$.
%defined on a stochastic basis $(\Omega, \Fscr, (\Fscr_t)_{t\geq 0}, \mathbb{P})$.
Then, the radial process $\|z-X_t\|$ is a $d$-dimensional Bessel process, independently of the choice of the Minkowski structure~$Q$.
\end{proposition*}

As a consequence of this proposition, any solution $X$ does not hit the singular point~$z$ due to the fact that the $d$-dimensional Bessel process does not hit zero for $d \ge 2$. From there, we conclude our pathwise uniqueness result; see the proof of Proposition \ref{prop:pathwise-unique} below.

In particular, this pathwise uniqueness result yields uniqueness for the linearized PDE \eqref{eqI:lin-FPE} in the class of \emph{all} distributional probability solutions. 
This is stronger than what is needed to apply Lemma~\ref{lem:equiv-extrem-uniqu} from \cite{RR25}, for which it suffices to prove uniqueness only within the class of solutions that are dominated by \(f_Q^z\).
The fact that the corresponding unique solution path laws constitute a nonlinear Markov process then follows as a consequence of this uniqueness result.

Finally, using the Yamada--Watanabe theorem, in Proposition \ref{prop:strong-sol} we obtain that the solution $X^z$ to the McKean--Vlasov SDE~\eqref{intro:DDSDE} is the unique probabilistically \emph{strong} solution after any strictly positive time $s>0$, i.e., $(X^z_t)_{t \geq s}$ is adapted to the canonical filtration of $W_{s+t}-W_s$, where $W$ denotes the driving Brownian motion.
% In the first step, the equation is localized away from the singularity. 
% On each compact annulus excluding the singular point, the coefficients are sufficiently regular, so that the standard Gr\"onwall-type argument implies uniqueness up to the corresponding exit time. 
% In the second step, we prove that the singular point is in fact almost surely unattainable. 
% This is achieved by studying an appropriate radial process and deriving a Bessel-type estimate showing that the hitting probability of the singularity is zero. 
% Since the solutions never reach the singularity, the local uniqueness argument extends to the entire time interval, and pathwise uniqueness follows.

We note that $s>0$ is essential in our pathwise uniqueness result and that we cannot (and do not need to) extend the proof to $s=0$ (the latter corresponds to the initial datum $\delta_z$, but our pathwise uniqueness proof crucially depends on the absolute continuity of the initial condition; see Remark \ref{rm:s>0}).

\paragraph{Conclusion.} 
We construct a new family of stochastic processes $X^z$ as a natural probabilistic counterpart to the Finsler heat kernel associated with a large class of smooth uniformly convex Minkowski norms $\|\cdot\|=\sqrt{2Q(\cdot)}$ on $\R^d$. 
These processes are constructed as the unique weak solutions to the singular McKean--Vlasov stochastic differential equation \eqref{intro:DDSDE} with Finsler heat kernel as prescribed one-dimensional time marginal densities, they form a uniquely determined nonlinear Markov process and are actually probabilistically strong solutions after any strictly positive time. Moreover, for any $\|\cdot\|$ as above, the path law of $\|z-X^z\|$ coincides with the path law of the standard $d$-dimensional Bessel process $|W|$, where $W$ denotes a standard $d$-dimensional Brownian motion.

\paragraph{Literature comparison.} 
In a series of recent papers, nonlinear Markov processes with one-dimensional time marginal densities given by (fundamental) solutions to nonlinear PDEs have been constructed for the porous medium \cite{RR25}, $2D$ vorticity Euler \cite{Rehmeier20242DVE} and Navier--Stokes \cite{BRZ23}, $p$-Laplace \cite{R.BR24-pLapl}, Leibenson \cite{BarbuGrubeRehmeierRoeckner2025Leibenson}, and nonlinear perturbed Ornstein--Uhlenbeck equations \cite{Rehmeier23weighted}. 
In each case, the nonlinear Markov process consists of unique solution path laws to a McKean--Vlasov stochastic differential equation associated with the given nonlinear PDE. 
The techniques used in the present work are substantially different from those in these papers, where the construction of nonlinear Markov processes with prescribed one-dimensional time marginal densities relied on analytic proofs of uniqueness results for associated linear PDEs, instead of the stronger probabilistic pathwise uniqueness result beyond the subcritical regime developed in the present work.

\paragraph{Organization.}
In Section \ref{sect:Finsler-stuff}, we review some properties of $Q$, and show that $f^z_Q$ solves the associated nonlinear Fokker--Planck equation.
Section \ref{sect:McKean-V-SDE} introduces the associated McKean--Vlasov stochastic differential equation and our first main result (Theorem \ref{thm:construction}), while Section \ref{sect:NLMP-stuff} recalls the notion of nonlinear Markov process and contains our second main result (Theorem \ref{thm:main-result}) as well as our definition of Brownian motion in Minkowski normed spaces (Definition \ref{def:FBm}). 
In Section \ref{sect:proof}, we prove Theorem \ref{thm:main-result}, which in particular amounts to proving the pathwise uniqueness result of Proposition \ref{prop:pathwise-unique}. We also formulate our second main result, Theorem \ref{thm:universality}.
In Section \ref{sect:strong}, we prove that the weak solutions constructed in Theorem \ref{thm:construction} are strong after any strictly positive time.
Finally, in Section \ref{sec:E}, we present some examples of Minkowski normed spaces to which our results apply.

\subsection{Notation}
We denote by $S^{d-1} \subset \R^d$ the unit sphere in $\R^d$, by $I$ the $(d \times d)$-identity matrix, and by $A^T$ the transpose of a $d \times d$-matrix $A$.
%(whenever the dimension $d$ is clear from context). 
For $x,y \in \R^d$, we denote by $|x|$ and $x \cdot y$ the Euclidean norm and inner product, respectively.
For a differentiable function $f$ on $\R^d$, we write $\nabla f:=(\partial_1 f, \ldots, \partial_d f)$ (instead of the Finsler notion of gradient vector).

We write $\delta_z$ for the Dirac measure at $z\in \R^d$ and $\Pscr$ for the set of Borel probability measures on $\R^d$. As already used in the introduction, we denote by $\mathcal{L}(X)$ the distribution of an $\mathbb{R}^d$-valued or path-valued random variable, i.e., in the latter case, $\mathcal{L}(X)$ denotes the path law of a stochastic process $X$ with paths in $C([0,\infty);\mathbb{R}^d)$.

For an open set $U \subset \R^d$ and $k \in \N \cup \{\infty\}$, we denote by $C^k(U)$ ($C^k_c(U)$, respectively) the usual spaces of $k$-times continuously differentiable functions $\varphi: U \to \R$ (of compact support, respectively), and by $C^k(U;\R^d)$ the space of $k$-times continuously differentiable vector fields $\varphi: U \to \R^d$. 
We write $D^2\varphi$ for the Hessian of a twice differentiable function $\varphi: \R^d \to \R$. 
For a measure space $(\Omega, \Fscr, \mu)$ and $k \in \N$, we write $L^k(\Omega; \mu)$ for the usual space of measurable functions $g: \Omega \to \R$ with $\int_{\Omega} |g|^k \,\dd\mu < \infty$, and $L^p_{\textup{loc}}(\Omega; \mu)$ for the corresponding spaces of locally integrable functions.

Throughout, we use Einstein summation convention.

\endgroup

\paragraph{Acknowledgements.}
SO is supported by the JSPS Grant-in-Aid for Scientific Research (KAKENHI) 22H04942, 24K00523, 24K21511, 26H01996, and by the JST CREST JPMJCR25Q2.
He is also grateful to Universit\"at Bonn for its hospitality during his visit in Summer 2026, a part of this work was carried out there.

\section{Finsler heat equation as a nonlinear Fokker--Planck equation}\label{sect:Finsler-stuff}

\subsection{Minkowski normed spaces}\label{sect.2.1}

Let $d\geq 2$ and assume throughout that $Q: \R^d \to [0,\infty)$ satisfies the following properties (the same as already listed in the introduction).
\begin{enumerate}[(H1)]
\item
$Q$ is positively $2$-homogeneous:
\[
Q(cx)=c^2 Q(x), \qquad \forall c>0,\ x\in\R^d;
\]
\item
$Q\in C^1(\R^d)\cap C^3(\R^d\setminus\{0\})$; 
\item
$\exists 0<\lambda\le \Lambda<\infty$ such that for every $x \in \R^d \setminus \{0\}$, as quadratic forms,
\[
\lambda I\le D^2Q(x)\le \Lambda I.
\]
\end{enumerate}

Under (H1)--(H3), $\|\cdot\|:=\sqrt{2Q}$ provides a uniformly convex and smooth Minkowski norm (see, e.g., \cite{Ohta2021}). We call $(\R^d, \|\cdot\|)$ a \emph{Minkowski normed space}. 
Regarding (H3), the existence of an upper constant~$\Lambda$ follows from the $C^2$-regularity of $Q$; thus the existence of a positive lower bound $\lambda$ (i.e., the uniform convexity of $Q$) is the actual assumption in (H3). 
See Section~\ref{sec:E} for some examples. 
Note that the positive definite matrix $D^2 Q$ gives the fundamental tensor in Finsler geometry (see \cite{BCS,Ohta2021}) and coincides with the canonical inner product when $\|\cdot\|$ is the standard Euclidean norm.
Since $D^2Q$ is positively $0$-homogeneous ($D^2Q(cx)=D^2Q(x)$ for $c>0$) as seen in the next lemma, we can equivalently assume (H3) only for all $x \in S^{d-1}$.

% \begin{example}
%   $N(x) = ||x||$, \cyan{KS. I know a couple of non-trivial examples. I will write about them. It would be also good to ask Shin-chi.}
% \end{example}

\begin{lemma}[Properties of $Q$]
\label{lem:structure} 
Under \textnormal{(H1)--(H3)}, the following hold.
\begin{enumerate}[(i)]%[label=\textnormal{(\alph*)}]
\item $D^2Q$ is positively $0$-homogeneous on $\R^d\setminus\{0\}$, and $A$ defined as 
\[
A(x):=\bigl( D^2Q(x)\bigr)^{-1} \quad (x\neq 0), \qquad  A(0):= I,
\]
is bounded, Borel measurable, symmetric and uniformly elliptic: 
\[
\frac{1}{\Lambda}I\le A(x)\le \frac{1}{\lambda}I
\qquad \forall x \in \R^d\setminus \{0\};
\]
\item
$x\cdot\nabla Q(x)=2Q(x)$ for all $x \in\R^d$;
\item
We have
\begin{equation}\label{eq:AgradQ}
A(x)\nabla Q(x)=x,
\qquad \forall x \in \R^d\setminus \{0\};
\end{equation}
\item
For $B = (B_1,\dots,B_d)$ with $B_i(x):= \partial_j A_{ij}(x)$, $x \in \R^d \setminus \{0\}$, there exists $C>0$ such that
\begin{equation}\label{eq:b-sing}
|B(x)|\le \frac{C}{|x|}
\qquad \forall x \in \R^d\setminus \{0\}.
\end{equation}
\end{enumerate}
\end{lemma}

\begin{proof}
The symmetry and measurability of $A$ are obvious.
Since $Q$ is positively $2$-homogeneous,  $D^2Q$ is positively $0$-homogeneous on $\R^d\setminus\{0\}$.
Combined with~(H3), this gives the claimed uniform ellipticity.

By differentiating the $2$-homogeneity~$Q(cx)=c^2Q(x)$ in the radial parameter~$c$ at~$c=1$, we have 
\[
x\cdot\nabla Q(x)=2Q(x), \qquad x \neq 0,
\]
and this identity extends to $x=0$ by
continuity.
Differentiation then yields
\[
D^2Q(x)x=\nabla Q(x),
\qquad x\ne 0,
\]
which implies \eqref{eq:AgradQ}.
Finally, since $A\in C^1(\mathbb{R}^d \setminus \{0\})$ (by $Q \in C^3(\mathbb R^d \setminus \{0\})$) and is positively $0$-homogeneous on $\R^d\setminus\{0\}$, its first derivatives are positively $(-1)$-homogeneous on $\R^d\setminus\{0\}$.
Hence, $|\nabla A(x)|\le C|x|^{-1}$ for some $C>0$, which gives \eqref{eq:b-sing}.
\end{proof}

\begin{remark} \label{r:BS}  
In the case of the standard Euclidean norm, since $D^2Q(x)$ is independent of $x$, one has $B\equiv 0$. 
In general, however, $B$ does not vanish (see Example \ref{ex:Randers}(a) for an example) and then, by the $(-1)$-homogeneity, it is as singular as $|B(x)| \sim 1/|x|$. 
\end{remark}

% \begin{remark}
% We stress that $B(x)=0$ for the Euclidean norm
% For example, if $d=2$ and $\partial_j A_{ij} \equiv 0$ holds on $\R^2 \setminus \{0\}$ for all $i,j=1,2$, then $A_{ij}$ is constant on $\R^2 \setminus \{0\}$.
% This implies that $D^2Q$ is constant on $\R^2 \setminus \{0\}$, and thus $Q$ necessarily comes from an inner product.
% \end{remark}

Denote by $Q^*$ the Legendre dual of $Q$, i.e.,
\[
Q^*(\xi) := \sup_{x\in \R^d}\{\xi \cdot x - Q(x)\},\quad \xi \in \R^d.
\]
%Recall that $2$-homogeneity as well as the regularity from (H2) is inherited from $Q$ to $Q^*$. Consequently, $\partial^2_{ij} Q$ and $\partial_{ij}^2 Q^*$ are $0$-homogeneous for all $1\leq i,j \leq d$, i.e.
%\begin{equation}
 %   \partial_{ij} Q(\lambda x) = \partial_{ij} Q(x), \quad\quad \partial_{ij} Q^*(\lambda x) = \partial_{ij} Q^*(x),\quad \forall \alpha \geq 0, x \in \R^d.
%\end{equation}
%Hence, for any $x\in \R^d\setminus \{0\}$,
%$D^2Q(x) = D^2Q(|x|\frac{x}{|x|}) = %D^2Q(\frac{x}{|x|})$. With (H3), this %implies
%\begin{equation}
%    \lambda I \leq D^2Q(x) \leq \Lambda I, \quad \forall x \in \R^d \setminus \{0\}.
%\end{equation}

\subsection{Finsler heat equation}

With $Q$ and $\|\cdot\|$ as above, the \emph{Finsler heat equation} on $(\R^d,\|\cdot\|)$ is given as in \eqref{eq:FLap}--\eqref{eq:FHE}:
%\cyan{KS. Having $\frac12$ is natural in relation to the standard Brownian motion. I will carefully check if $\frac12$ is not missing in the rest, but please also check it, Marco and Shin-ichi, by yourselves once more before the submission.}
\begin{align} 
   \notag  \partial_t u &= \frac 1 2 \Delta u\\ 
   \notag  &= \frac 1 2 \partial_i \bigl( \partial_i Q^*(\nabla u)\bigr)
     \\& \label{FT}= \frac 1 2 \partial_i \bigl(  \partial_{ij}Q^*(\nabla u) \partial_j u \bigr).
\end{align}
The final equality holds, since the 1-homogeneity of $\partial_i Q^*$ and~(ii) of Lemma~\ref{lem:structure} imply 
\begin{equation}\label{x}
    \partial_i Q^*(\xi) = \xi_j \partial_{ij} Q^*(\xi),\quad \forall \xi \in \R^d \setminus \{0\}.
\end{equation} 
We refer to \cite{Ohta2021,OhtaSturm2009} for the Finsler heat equation on general (measured) Finsler manifolds (without the coefficient $1/2$).

The fundamental solution to \eqref{FT} (with initial point source $z\in \R^d$) is the \emph{Finsler heat kernel}
\begin{equation}\label{eq:def-kernel}
f_Q^z(t,x):=\frac{1}{(2t)^{d/2}|B_Q|\Gamma(1+\frac{d}{2})}
\exp\biggl(-\frac{Q(z-x)}{t}\biggr),\quad (t,x) \in (0,\infty)\times \R^d,
\end{equation}
with $B_Q:=\{x\in\R^d \,|\, Q(x)<1/2\}$, which has finite Lebesgue measure $|B_Q| < \infty$ (see \cite[Example~4.3]{OhtaSturm2009}).
We collect some properties of $f^z_Q$ used in the sequel.

\begin{lemma}[Properties of $f^z_Q$] \ 
\label{lem:props-f}
    \begin{enumerate}[(i)]
        \item $f^z_Q(t) := f^z_Q(t,\cdot)$ is a probability density for all $t>0$, $t\mapsto f^z_Q(t,x)\,\dd x$ is weakly continuous on $(0,\infty)$ in the sense of measures, and $f^z_Q(t,x) \,\dd x \to \delta_z$ weakly as $t \to 0$.
        \item $f^z_Q(t) \in C^1(\R^d)$ and $\nabla f^z_Q(t,x) = \frac{1}{t} f^z_Q(t,x) \nabla Q(z-x)$ for all $t>0$.
        \item Let $\varphi \in C^1_c(\R^d)$ and $t>0$. Then
        \begin{align*}
            &\int_{\R^d} \varphi(x) f^z_Q(t,x) \,\dd x - \varphi(z) 
            \\&= -\frac 1 2\int_0^t \int_{\R^d}  \nabla Q^*\bigl(\nabla f^z_Q(r,x)\bigr) \cdot \nabla \varphi(x) \,\dd x \dd r \\
            &= -\frac 1 2 \int_0^t \int_{\R^d}  \partial_{ij}Q^* \bigl(\nabla f^z_Q(r,x)\bigr) \partial_j f^z_Q(r,x) \partial_i \varphi(x) \,\dd x \dd r.
        \end{align*}
    \end{enumerate}
\end{lemma}

\begin{proof}
\begin{enumerate}[(i)]
    \item $\int_{\R^d} f^z_Q(t,x) \,\dd x = 1$ for all $t>0$ follows from a straightforward calculation, while the claimed weak continuity on $(0,\infty)$ follows for instance from the continuity of $f^z_Q$ in $(t,x)$.
    The final claim about weak convergence as $t\to 0$ can be proven analogously to the proof that a standard normal distribution with covariance $t I$ converges to $\delta_0$ in the sense of distributions as $t \to 0$.
    \item Both claims follow directly from the definition of $f^z_Q$.
    \item The first identity follows from the fact that $f^z_Q$ is a weak solution to \eqref{eq:FHE}, and the second one from \eqref{x}. \qedhere
\end{enumerate}
\end{proof}

\subsection{Associated nonlinear Fokker--Planck equation}

We introduce coefficients $a = (a_{ij})_{1 \leq i,j \leq d}$, $b = (b_1,\dots,b_d)$, defined on $\R^d \times \mathfrak{C}$ (the latter is defined below),
\begin{equation}\label{coeffs-A,B}
    \begin{aligned}
    a_{ij}(x,\varphi) &:= \frac 1 2\partial_{ij} Q^*\bigl(\nabla \varphi (x)\bigr), \\
    b_i(x,\varphi) &:= \partial_j a_{ij}(x,\varphi) = \frac 1 2 \partial_j \Bigl( \partial_{ij}Q^*\bigl(\nabla \varphi(x)\bigr)\Bigr),
    \end{aligned}
\end{equation}
with $a_{ij}(x,\varphi) := \frac{1}{2} \delta_{ij}$ and $b_i(x,\varphi) := 0 $ for $(x,\varphi)$ with $\nabla \varphi(x) = 0$ (this is necessary since $\partial_i Q^*$ is not differentiable at $0$). 
Here, 
\[
\mathfrak{C} := \{\varphi \in C^1(\R^d) \,|\, \nabla \varphi \in C^1(\R^d\setminus O_{\nabla \varphi};\R^d)\}, 
\]
and we define the closed set $O_{\nabla \varphi} := \{x \in \R^d \,|\, \nabla \varphi (x) = 0\}$. 
Since $a(x,\varphi)$ is symmetric, we have $b(x,\varphi) = \divv a(x,\varphi).$

\begin{remark}\label{rem:f-in-C}
$f^z_Q(t) \in \mathfrak{C}$ for all $z\in \R^d$ and $t >0$. 
This follows from Lemma \ref{lem:props-f}(ii), $O_{\nabla Q}= \{0\}$, and $\nabla Q \in C^2(\R^d\setminus \{0\};\R^d)$.
\end{remark}

The next lemma shows that, for $\varphi = f^z_Q(t)$, the coefficients in \eqref{coeffs-A,B} take a particular appealing form independent of $t>0$.

\begin{lemma}[Properties of $a_{ij},b_i$]
\label{lem:nice-form}
Let $1\leq i, j \leq d$, $z\in \R^d$ and $t>0$.
    Then
    \[
    a_{ij}\bigl(x,f^z_Q(t)\bigr) = \frac 1 2 \partial_{ij}Q^*\bigl(\nabla Q(z-x)\bigr),\quad \forall x \in \R^d \setminus \{z\},
    \]
    and $a_{ij}(z,f^z_Q(t)) = \frac{1}{2}\delta_{ij}$, i.e.,
    \begin{equation}\label{eq:aAbB}
        a_{ij}\bigl(x,f^z_Q(t)\bigr) = \frac{1}{2} A_{ij}(z-x), \qquad b_i\bigl(x,f^z_Q(t)\bigr) = -\frac{1}{2} B_i(z-x),
    \end{equation}
    where $A$ and $B$ are as in Lemma \ref{lem:structure}.
    Moreover, $(t,x) \mapsto a_{ij}(x,f^z_Q(t))$ is bounded on $(0,\infty)\times \R^d$.
\end{lemma}

\begin{proof}
Since $\partial_{ij} Q^*$ is $0$-homogeneous, for $x \neq z$, we find
\begin{align*}
a_{ij}\bigl(x,f^z_Q(t)\bigr)
&= \frac 1 2 \partial_{ij}Q^*\bigl(\nabla f^z_Q(t,x)\bigr)
= \frac 1 2 \partial_{ij}Q^*\biggl(\frac 1 {t}f^z_Q(t,x) \nabla Q(z-x)\biggr) \\
&= \frac 1 2 \partial_{ij}Q^*\bigl(\nabla Q(z-x)\bigr).
\end{align*}
For $x=z$, $\nabla f^z_Q(t,z) = 0$; thus, by definition, $a_{ij}(z,f^z_Q(t)) = \frac{1}{2}\delta_{ij}$ as claimed.
Then \eqref{eq:aAbB} follows from the 2-homogeneity of $Q$ (which implies $D^2Q^*(\nabla Q(y)) = (D^2Q(y))^{-1}$ for all $y\neq 0$), and Lemma \ref{lem:structure}(i) implies the boundedness of $a_{ij}(x,f^z_Q(t))$.
\end{proof}

We now associate the Finsler heat equation~\eqref{FT} with the \emph{nonlinear Fokker--Planck equation}
\begin{equation}\label{eq:NL-FPE}
    \partial_t u(t,x) = \partial_{ij}\Bigl(a_{ij}\bigl(x,u(t)\bigr) u(t,x) \Bigr)   -\partial_i \Bigl(b_i\bigl(x,u(t)\bigr) u(t,x) \Bigr)
\end{equation}
for $(t,x)\in (0,\infty)\times \R^d$, with $a$ and $b$ as above.

\begin{definition}[Solutions to the nonlinear Fokker--Planck equation]\label{def:F-FPE}
Let $z\in \R^d$.
    We say that $u: (0,\infty)\times \R^d \to [0,\infty)$ is a \emph{probability solution} to \eqref{eq:NL-FPE} with initial datum $z$ if $u(t):=u(t,\cdot) \in \mathfrak{C}$ for all $t >0$ and
    \begin{enumerate}[(i)]
        \item $u(t,x) \,\dd x \in \Pscr$ for all $t>0$, $t\mapsto u(t,x) \,\dd x$ is weakly continuous on $(0,\infty)$, and $u(t,x) \,\dd x \to \delta_z$ weakly as $t \to 0$;
        \item For each $T>0$, the maps $(t,x) \mapsto a_{ij}(x,u(t))$ and $(t,x) \mapsto b_i(x,u(t))$ belong to $L^1([0,T]\times \R^d; u(t,x) \,\dd x \dd t)$;
        \item For all $\varphi \in C^2_c(\R^d)$ and $t>0$,
        \begin{align*}
            &\int_{\R^d} \varphi(x) u(t,x) \,\dd x - \varphi(z) 
            \\&= \int_0^t \int_{\R^d} \Bigl(a_{ij}\bigl(x,u(r)\bigr) \partial_{ij}\varphi(x) + b_i\bigl(x,u(r)\bigr) \partial_i \varphi(x) \Bigr) u(r,x) \,\dd x \dd r.
        \end{align*}
    \end{enumerate}
\end{definition}

This definition is consistent with the definition of solutions to general nonlinear Fokker--Planck equations, for instance as in \cite{RR25}. 

Formally, the Finsler heat equation \eqref{FT} and the nonlinear Fokker--Planck equation \eqref{eq:NL-FPE} are equivalent, as a straightforward integration by parts on the RHS of \eqref{FT} with a test function $\varphi$ shows.
However, rigorously proving this equivalence is delicate.
For our purpose, the following result is sufficient.

\begin{lemma}\label{lem:f-solves-NLFPE}
For any $z\in \R^d$, $f^z_Q$ is a probability solution to \eqref{eq:NL-FPE} with initial datum $z$.
\end{lemma}

\begin{proof}
Thanks to Lemma \ref{lem:props-f}(i) and Remark \ref{rem:f-in-C}, it remains to prove (ii) and (iii) of Definition \ref{def:F-FPE} for $u=f^z_Q$.
First, note that the Borel measurability of $a_{ij}(x,f^z_Q(t))$ and $b_i(x,f^z_Q(t))$ in $(t,x)$ follows from Lemma \ref{lem:nice-form}. 
Secondly, since $(t,x) \mapsto a_{ij}(x,f^z_Q(t))$ is bounded on $(0,\infty)\times \R^d$ and $\int_0^T \int_{\R^d} f^z_Q(t,x) \,\dd x \dd t = T$ for each $T>0$, $(t,x) \mapsto a_{ij}(x,f^z_Q(t))$ belongs to $L^1([0,T]\times \R^d; f^z_Q(t,x) \,\dd x \dd t)$.
Moreover, Lemma \ref{lem:structure}(iv) yields $|b_i(x,f^z_Q(t))| \leq C|z-x|^{-1}$ for all $x \neq z$ and $t >0$, and hence the map $(t,x) \mapsto b_i(x,f^z_Q(t))$ belongs to $L^1([0,T]\times \R^d; f^z_Q(t,x) \,\dd x \dd t)$ as well, since in dimension $d\geq 2$, we have 
\[
\int_{\R^d} |z-x|^{-1} f^z_Q(t,x) \,\dd x \le C t^{-\frac 1 2},
\]
where $C>0$ depends on $Q$, but not on $t>0$.
Finally, let $\varphi \in C^2_c(\R^d)$ and $t>0$. 
Then, using Lemma \ref{lem:props-f}(iii) and integration by parts, we obtain by the definition of $a_{ij}$ and $b_i$,
 \begin{align*}
            &\int_{\R^d} \varphi(x) f^z_Q(t,x) \,\dd x - \varphi(z) 
            \\&= -\frac 1 2 \int_0^t \int_{\R^d}  \partial_{ij}Q^* \bigl(\nabla f^z_Q(r,x)\bigr) \partial_j f^z_Q(r,x) \partial_i \varphi(x) \,\dd x \dd r
            \\&= \int_0^t \int_{\R^d} \Bigl(a_{ij}\bigl(x,f^z_Q(r)\bigr) \partial_{ij}\varphi(x) + b_i\bigl(x,f^z_Q(r)\bigr) \partial_i \varphi(x) \Bigr) f^z_Q(r,x) \,\dd x \dd r.
        \end{align*}
Thus, $f^z_Q$ satisfies (iii) of Definition \ref{def:F-FPE}, which completes the proof. 
\end{proof}

\subsection{Linearized equation}

Fix $s\geq 0$ and consider the \emph{linear} Fokker--Planck equation
\begin{equation}\label{eq:lin-FPE}
     \partial_t u(t,x) = \partial_{ij}\Bigl(a_{ij}\bigl(x,f^z_Q(s+t)\bigr) u(t,x) \Bigr)   -\partial_i \Bigl(b_i\bigl(x,f^z_Q(s+t)\bigr) u(t,x) \Bigr)
\end{equation}
for $(t,x)\in (0,\infty)\times \R^d$,
obtained by a priori fixing $f^z_Q(s+t)$ in the second arguments of $a_{ij}$ and $b_i$. The time shift by $s\geq 0$ will become relevant later on.
Equation \eqref{eq:lin-FPE} is a linear Fokker--Planck equation with time-dependent coefficients
\begin{equation}\label{eq:lin-coeff}
a^{s,z}_{ij}(t,x) := a_{ij}\bigl(x,f^z_Q(s+t)\bigr), \quad b^{s,z}_i(t,x) := b_i\bigl(x,f^z_Q(s+t)\bigr).
\end{equation}
This kind of linearization has been studied on general Finsler manifolds in \cite{OhtaSturm2014} to develop a nonlinear analog to the $\Gamma$-calculus technique (see also \cite{Ohta2017,Ohta2021,Ohta2022,OhtaSuzuki2025} for further applications).

%see Definition \ref{def:lin-FPE}. For the reader's convenience, we give the definition in this particular case $a_{ij} = a^{s,z}_{ij}$ and $b_i = b^{s,z}_i$ here:

\begin{definition}[Solutions to the linear Fokker--Planck equation]
\label{def:lin-FPE}
    Let $z\in \R^d$ and $s\geq 0$. 
    We say that $u: (0,\infty)\times \R^d \to [0,\infty)$ is a \emph{probability solution} to \eqref{eq:lin-FPE} with initial datum $\zeta \in \Pscr$ if
    \begin{enumerate}[(i)]
        \item $u(t,x) \,\dd x \in \Pscr$ for all $t>0$, $t\mapsto u(t,x) \,\dd x$ is weakly continuous on $(0,\infty)$, and $u(t,x) \,\dd x \to \zeta$ weakly as $t \to 0$;
        \item The maps $(t,x) \mapsto a_{ij}(x,f^z_Q(s+t))$ and $(t,x) \mapsto b_i(x,f^z_Q(s+t))$ belong to $L^1([0,T]\times \R^d; u(t,x) \,\dd x \dd t)$ for each $T>0$;
        \item For all $\varphi \in C^2_c(\R^d)$ and $t>0$,
        \begin{align*}
            &\int_{\R^d}\varphi(x) u(t,x) \,\dd x - \int_{\R^d}\varphi(x) \,\dd\zeta(x)
            \\&= \int_0^t \int_{\R^d} \Bigl(a_{ij}\bigl(x,f^z_Q(s+r)\bigr) \partial_{ij}\varphi(x) + b_i\bigl(x,f^z_Q(s+r)\bigr) \partial_i \varphi(x) \Bigr) u(r,x) \,\dd x \dd r.
        \end{align*}
    \end{enumerate}
\end{definition}

%We stress that a solution in this definition is defined on $(0,\infty)$ and its initial datum is attained at $t=0$, since the time shift by $s\geq 0$ only affects

%The time shift by $s\geq 0$ can change the curve $t\mapsto f^z_Q(s+t)$ which we a priori fix in the arguments of $a$ and $b$, but it does not change the time horizon $[0,\infty)$ on which \eqref{eq:lin-FPE} is solved.

% \red{M: I want to clarify that here by "time shift" we do NOT mean that we solve the equation from a strictly positive time on. The initial time from which we solve the equation is still $0$, i.e., solutions are defined on $[0,\infty)$, not on $[s,\infty)$. The shift affects (only) the time-variable in the frozen coefficient. I improved your grammar point.}

We make the following immediate, yet important observation.

\begin{lemma}\label{lem:f-solves-linfpe}
    Let $z\in \R^d$ and $s\geq 0$. 
    Then $u(t,x) = f^z_Q(s+t,x)$ is a probability solution to \eqref{eq:lin-FPE} with initial datum $\zeta = f^z_Q(s,x) \,\dd x$ \textnormal
    {(}here, with slight abuse of notation, we write $f^z_Q(0,x) \,\dd x = \delta_z$ if $s=0)$.
\end{lemma}

\begin{proof}
    This follows directly from Lemma \ref{lem:f-solves-NLFPE}.
\end{proof}

In general, \eqref{eq:lin-FPE} may have many solutions which are not related to $f^z_Q$ at all. The previous lemma states that, in any case, one particular solution is given by $u(t,x) = f^z_Q(s+t,x)$.

\section{The associated McKean--Vlasov equation}\label{sect:McKean-V-SDE}
In this section, we study the following McKean--Vlasov SDE:
\begin{equation}\label{eq:DDSDE}
\left\{
\begin{aligned}
    \dd X_t &= b\bigl(X_t,u(t)\bigr) \,\dd t + \sqrt{2a \bigl(X_t,u(t)\bigr)} \,\dd W_t
    \\ &=\frac 1 2 \divv \Bigl(D^2Q^*\bigl(\nabla u(t,X_t)\bigr)\Bigr) \,\dd t + \sqrt{D^2Q^*\bigl(\nabla u(t,X_t)\bigr)} \,\dd W_t,
    \\ \mathcal{L}(X_t) &= u(t,x) \,\dd x,\ \forall t>0,
\end{aligned}
\right.
\end{equation}
where $a,b$ are as in \eqref{coeffs-A,B}, $W = (W_t)_{t \geq 0}$ is the standard $d$-dimensional Brownian motion, and $\mathcal{L}(X_t)$ denotes the distribution of $X_t$.

This SDE is the natural corresponding probabilistic model for the Finsler heat equation, more precisely for its associated nonlinear Fokker--Planck formulation \eqref{eq:NL-FPE}. Indeed, it is readily seen from Itô's formula that the curve of one-dimensional time marginals $\mu_t = \mathcal{L}(X_t)$ of any solution $X$ to \eqref{eq:DDSDE} solves \eqref{eq:NL-FPE}. The goal of this section is the substantially more delicate reverse implication, i.e., that for a given solution $u$ to \eqref{eq:NL-FPE} consisting of probability densities $u(t)$, there exists a solution to \eqref{eq:DDSDE} with $\mathcal{L}(X_t) = u(t,x) \,\dd x$ for all $t\geq 0$. More precisely, we will construct such an SDE-solution with one-marginal densities $u(t) = f^z_Q(t)$.

%\red{M: change these first sentences} As seen in the first paragraph of Introduciton, for a given linear Fokker-Planck (FP) equation we have the corresponding stochastic differential equations whose time-marginal densities are given by the solutions to the liner FP equation.  A natural question would be whether it is the case also for our {\it nonlinear} FP equation. 

%We will see in Proposition~\ref{prop:SP} (\emph{superposition principle}) with the aid of Lemma \ref{lem:nice-form} that Equation~\eqref{eq:DDSDE} is indeed the right SDE in the sense that the time-marginal densities are given by the solution to the nonlinear Fokker-Planck equation \eqref{eq:NL-FPE}.
% The reason why $\eqref{eq:DDSDE}$ is \emph{naturally associated} to \eqref{eq:NL-FPE} is that both equations are related via the \emph{superposition principle}, see~ below.
%We are going to construct a probabilistic counterpart for the Finsler heat kernel $f^z_Q$ by constructing unique solutions to \eqref{eq:DDSDE} with $u(t)$ replaced by $f^z_Q(t)$.

\begin{definition}[Solutions to the McKean--Vlasov SDE]
\label{def:F-DDSDE} \ 
\begin{enumerate}[(i)]
    \item A (\emph{probabilistically weak}) \emph{solution} to \eqref{eq:DDSDE} is a quadruple consisting of a stochastic basis $(\Omega, \Fscr, (\Fscr_t)_{t\geq 0}, \mathbb{P})$, an $(\Fscr_t)$-standard Brownian motion $W = (W_t)_{t\geq 0}$, a family of probability densities $u(t) \in \mathfrak{C}$, $t>0$,  and an $(\Fscr_t)$-adapted stochastic process $X = (X_t)_{t\geq 0}$ on $\Omega$ such that $\mathcal{L}(X_t)  = u(t,x) \,\dd x$ for all $t>0$,
    \begin{equation*}
        \mathbb{E}\bigg[ \int_0^T \Bigl( \bigl|b\bigl(X_t,\mathcal{L}(X_t)\bigr)\bigr| + \bigl|a\bigl(X_t,\mathcal{L}(X_t)\bigr) \bigr| \Bigr) \,\dd t \bigg] < \infty, \quad \forall T>0
    \end{equation*}
    (where $\mathbb{E}$ denotes the expectation with respect to $\mathbb{P}$),
    and, $\mathbb{P}$-a.s.,
    \begin{equation*}
        X_t -X_0 = \int_0^t b\bigl( X_r,\mathcal{L}(X_r) \bigr) \,\dd r + \int_0^t \sqrt{2a\bigl(X_r,\mathcal{L}(X_r)\bigr)} \,\dd W_r,\quad \forall t\geq 0.
    \end{equation*}
Below, we usually simply refer to $X$ instead of the entire quadruple as a solution.
\item For each solution $X$, we call its path law $P = \mathcal{L}(X)$ (a probability measure on $C([0,\infty);\R^d)$) a \emph{solution path law} to \eqref{eq:DDSDE}.
\item We say that solutions to \eqref{eq:DDSDE} with initial datum $\zeta \in \Pscr$ are \emph{weakly unique} if $\mathcal{L}(X) = \mathcal{L}(Y)$ for any two solutions $X,Y$ with $\mathcal{L}(X_0) = \mathcal{L}(Y_0) = \zeta$. %Equivalently, solution path laws with initial condition $\zeta$ are unique.
\end{enumerate}
\end{definition}

If we prescribe $u$ in \eqref{eq:DDSDE} as $u(t,x) \,\dd x = f^z_Q(t,x) \,\dd x$, then by Lemma \ref{lem:nice-form}, \eqref{eq:DDSDE} turns into
\begin{equation}\label{eq:f-DDSDE}
\left\{
\begin{aligned}
    \dd X_t &= \frac 1 2 \divv \Bigl(D^2 Q^*\bigl(\nabla Q(z-X_t) \bigr)\Bigr) \,\dd t + \sqrt{D^2Q^*\bigl(\nabla Q(z-X_t) \bigr)} \,\dd W_t,
    \\ \mathcal{L}(X_t) &= f^z_Q(t,x) \,\dd x,\ \forall t>0.
\end{aligned}
\right.
\end{equation}

The notion of solution to \eqref{eq:f-DDSDE} is that of probabilistically weak solution $X$ to the SDE with coefficients $b(x) = \frac 1 2 \divv (D^2Q^*(\nabla Q(z-x)))$ and $\sigma(x) = \sqrt{D^2Q^*(\nabla Q(z-x))}$ in the usual sense, satisfying in addition the marginal condition $\mathcal{L}(X_t) = f^z_Q(t,x) \,\dd x$.

As our first main result (Theorem \ref{thm:construction}), we construct weak solutions to \eqref{eq:DDSDE} with marginal densities $u(t,x) = f^z_Q(t,x)$.
The proof relies on the following celebrated Ambrosio--Figalli--Trevisan \emph{superposition principle} (see \cite[Theorem 2.5]{Trevisan16}).

\begin{proposition}[Superposition principle]
\label{prop:SP}
    Let $a^0 = (a^0_{ij})_{1 \leq i,j \leq d}$ and $b^0 = (b^0_1,\ldots,b^0_d)$ consist of Borel measurable coefficients 
    \[
    a^0_{ij}, b^0_i: (0,\infty)\times \R^d \to \R
    \]
    such that $a^0$ is pointwise symmetric and non-negative definite.
    Suppose that $u: (0,\infty)\times \R^d \to [0,\infty)$ is a probability solution to the linear Fokker--Planck equation
    \begin{equation}\label{eq:linfpe}
        \partial_t u(t,x) = \partial_{ij}\bigl(a^0_{ij}(t,x) u(t,x)\bigr) - \partial_i \bigl(b^0_i(t,x) u(t,x)\bigr)
    \end{equation}
with a probability measure-valued initial datum $\zeta$, satisfying
\begin{equation}\label{intint}
    a^0_{ij}, b^0_i \in L^1\bigl( [0,T]\times \R^d;u(t,x) \,\dd x \dd t \bigr),\quad \forall T>0.
\end{equation}
%in the sense of Definition \ref{def:lin-FPE}%
Then there exists a probabilistically weak solution $X = (X_t)_{t\geq 0}$ to the SDE
\begin{equation*}
    \dd X_t = b^0(t,X_t) \,\dd t + \sigma^0(t,X_t) \,\dd W_t,
\end{equation*}
%in the sense of Definition \ref{def:lin-SDE}
with $\sigma^0:(0,\infty)\times \R^d \to \R^{d\times d}$ such that $\sigma^0 (\sigma^0)^T = 2 a^0$, $\mathcal{L}(X_t) = u(t,x)\,\dd x$ for all $t>0$, and $\mathcal{L}(X_0) = \zeta$.
\end{proposition}

\begin{remark} 
    We stress that the previous proposition only applies to solutions $u$ to the linear Fokker--Planck equation \eqref{eq:linfpe} which satisfy the global in space integrability condition 
   \eqref{intint}
    (or a slightly more general, but still global condition; see \cite{BRS19-SPpr}). 
    To define solutions to \eqref{eq:linfpe} most generally, it suffices to require
    \[
    a^0_{ij}, b^0_i \in L^1_{\textup{loc}}\bigl([0,\infty)\times \R^d;u(t,x) \,\dd  x \dd t \bigr),
    \]
    but for such merely locally integrable solutions, the previous result does not apply in general. 
    Our definition of probability solutions to the linear Fokker--Planck equation \eqref{eq:lin-FPE} (see Definition \ref{def:lin-FPE}) includes this global condition, and we explicitly verified it for $u = f^z_Q$ in the proof of Lemma \ref{lem:f-solves-NLFPE}. 
    Hence, in the proof of Theorem \ref{thm:construction} below, we may indeed apply this superposition principle to $f^z_Q$.
\end{remark}

The following theorem is our first main result.

\begin{theorem}[Existence]
\label{thm:construction}
    Let $d\geq 2$, $z\in \R^d$. 
    Then there exists a probabilistically weak solution $X^z = (X^z_t)_{t\geq 0}$ to the McKean--Vlasov equation \eqref{eq:DDSDE} with $\mathcal{L}(X^z_t) = f^z_Q(t,x) \,\dd x$ for all $t>0$ and $X^z_0 = z$.
\end{theorem}

\begin{proof}
    By Lemma \ref{lem:f-solves-linfpe} with $s=0$, $f^z_Q$ is a probability solution to \eqref{eq:lin-FPE} with initial datum $\delta_z$. By Proposition \ref{prop:SP} applied to the linear Fokker--Planck equation with coefficients $a_{ij}^{0,z},b_i^{0,z}$ from \eqref{eq:lin-coeff}, there exists a weak solution $X^z = (X^z_t)_{t\geq 0}$ to \eqref{eq:f-DDSDE}. The integrability condition~\eqref{intint} has been already shown  in the proof of~Lemma~\ref{lem:f-solves-NLFPE}. Hence $X^z$ also solves \eqref{eq:DDSDE} with $\mathcal{L}(X^z_t) = f^z_Q(t,x) \,\dd x$, as claimed.
\end{proof}

\section{Construction of Brownian motion in Minkowski normed spaces}\label{sect:NLMP-stuff}

We turn to our main result (Theorem \ref{thm:main-result}), stating that the solutions $X^z$ to the McKean--Vlasov equation \eqref{eq:DDSDE} constructed in Theorem \ref{thm:construction} are weakly unique conditioned on their one-dimensional time marginal densities being $f^z_Q(t)$, and that these solutions form a nonlinear Markov process. 

\subsection{Nonlinear Markov processes and cores}

We recall some basic notions and key results from \cite{ABR26,RR25}. 
We refer to these papers for more details on nonlinear Markov processes and cores.

Below we use the following canonical model-notation:
$\Cscr := C([0,\infty);\R^d)$ is the space of continuous paths from $[0,\infty)$ to $\R^d$ endowed with the topology of locally uniform convergence,
$\pi_t: \Cscr \to \R^d$ is the evaluation map $\pi_t(w) := w(t)$,
$\Pi_t : \Cscr \to \Cscr$ is the shift $\Pi_t(w) := w(t+\cdot)$,
and $\Fscr_s := \sigma(\pi_r, 0\leq r \leq s)$ denotes the corresponding canonical filtration.
Then the Borel $\sigma$-algebra on $\Cscr$ is given by $\Bscr(\Cscr) = \sigma(\pi_r, r \geq 0)$ and we write $\Pscr(\Cscr)$ for the set of Borel probability measures on $\Cscr$.
For $P \in \Pscr(\Cscr)$, we denote by $P_t \in \Pscr$ the image measure (push-forward) of $P$ under $\pi_t$, i.e., $P_t = P\circ \pi_t^{-1}$.

\begin{definition}[Nonlinear Markov processes; see Definition 2.1 in \cite{RR25}]
\label{def:NL-MP}
   Let $\Pscr_0 \subseteq \Pscr$. 
   A (\emph{time-homogeneous}) \emph{nonlinear Markov process} is a family $\{P^\zeta\}_{\zeta \in \Pscr_0} \subseteq \Pscr(\Cscr)$ such that, setting $\mu^\zeta_t := P^\zeta_t \in \Pscr$ for $(t,\zeta) \in [0,\infty) \times \Pscr_0$,
    \begin{enumerate}[(i)]
        \item $\mu^\zeta_0 = \zeta$ for all $\zeta \in \Pscr_0$, and $\mu^\zeta_t \in \Pscr_0$ for all $(t,\zeta) \in (0,\infty)\times \Pscr_0$;
        \item the (\emph{time-homogeneous}) \emph{nonlinear Markov property} holds, i.e.,
        \begin{equation*}
        P^\zeta(\pi_{t+s} \in A | \Fscr_s) ( \cdot) = P^{\mu^\zeta_s,\pi_s(\cdot)}(\pi_t \in A)\quad P^\zeta\text{-a.s.},\quad \forall s,t \geq 0, \, \zeta \in \Pscr_0
    \end{equation*}
    for every Borel set $A \subseteq \R^d$,
    where for $\zeta \in \Pscr_0$, $\{P^{\zeta,y}\}_{y\in \R^d}$ is the disintegration family of $P^{\zeta}$ with respect to $\pi_0$, i.e., the $\zeta$-a.s.\ unique family $\{Q^y\}_{y\in \R^d}$ in $\Pscr(\Cscr)$ such that $y\mapsto Q^y(W)$ is measurable and $P^\zeta(W) = \int_{\R^d} Q^y(W) \,\dd \zeta(y)$ for all $W \in \Bscr(\Cscr)$.
    \end{enumerate}
\end{definition}

By definition, the set of initial data $\Pscr_0$ for a nonlinear Markov process is invariant under the flow of its one-dimensional time-marginals.
This excludes, except for deterministic cases, the choice $\Pscr_0 = \{\delta_z \,|\, z \in \R^d\}$ which, on the other hand, appears a natural choice if one wants to construct a nonlinear Markov process with one-dimensional time marginals given by the family of fundamental solutions of a nonlinear PDE.
This led to the following notion of \emph{nonlinear Markov core} from \cite{ABR26}.
The idea is to construct the \emph{minimal} nonlinear Markov process containing such one-dimensional time marginals. 

\begin{definition}[Nonlinear Markov cores]
\label{def:NL-MC}
 Let $\{\mu^z_t\}_{t\geq 0, z \in \R^d} \subseteq \Pscr$ be such that $\mu^z_0 = \delta_z$ for all $z \in \R^d$ and $(t,z) \mapsto \mu^z_t$ is injective from $[0,\infty)\times \R^d$ to $ \Pscr$. We call a family $\{P^z\}_{z\in \R^d} \subseteq \Pscr(\Cscr)$ a (\emph{time-homogeneous}) \emph{nonlinear Markov core} for $\{\mu^z_t\}_{t\geq 0, z \in \R^d}$ if
        \begin{enumerate}[(i)]
        \item $P^z_t = \mu^z_t$ for all $(t,z)\in  [0,\infty) \times \R^d$;
        \item the family $\{P^\zeta\}_{\zeta \in \mathcal{P}_0}\subseteq \Pscr(\Cscr)$, defined by $\Pscr_0 \coloneqq \{\mu^z_t\}_{t\geq 0, z \in \R^d}$ and 
        \begin{equation*}
            P^{\zeta} \coloneqq P^z \circ \Pi_t^{-1} \quad \text{for } \zeta = \mu^z_t,
        \end{equation*}
     is a nonlinear Markov process in the sense of Definition \ref{def:NL-MP}. 
    \end{enumerate}
    We call $\{P^\zeta\}_{\zeta \in \mathcal{P}_0}$ from (ii) the nonlinear Markov process \emph{induced by} $\{P^z\}_{z\in \R^d}$.
    If for each $(t,z) \in (0,\infty)\times \R^d$, $\mu^z_t$ is given by a density as $\mu^z_t = u^z(t,x) \,\dd x$, then we also call $\{P^z\}_{z\in \R^d}$ a nonlinear Markov core for $\{u^z(t)\}_{t>0, z\in \R^d}$.
\end{definition}

The previous definition implies that the nonlinear Markov process $\{P^\zeta\}_{\zeta \in \mathcal{P}_0}$ induced by a nonlinear Markov core $\{P^z\}_{z\in \R^d}$ is uniquely determined by the latter. 
Hence, even though a nonlinear Markov core does not satisfy Definition \ref{def:NL-MP}, it is equivalent to its induced nonlinear Markov process.

For the following results from \cite{ABR26} which we are going to use below, consider the following general nonlinear and linear Fokker--Planck equations and McKean--Vlasov SDE:
 \begin{equation}\label{gen-nl-fpe}
    \partial_t u(t,x) = \partial_{ij}\Bigl(a^0_{ij}\bigl(x,u(t)\bigr) u(t,x) \Bigr) - \partial_i \Bigl(b^0_i\bigl(x,u(t)\bigr) u(t,x)\Bigr),
    \end{equation}
       \begin{equation}\label{gen:lin-fpe}
       \partial_t u(t,x) = \partial_{ij}\Bigl(a^0_{ij}\bigl(x,v(t)\bigr) u(t,x) \Bigr) - \partial_i \Bigl(b^0_i \bigl(x,v(t)\bigr) u(t,x)\Bigr),
    \end{equation}
     \begin{equation}\label{gen-DDSDE}
   \left\{
   \begin{aligned}
   \dd X_t &= b^0\bigl(X_t,u(t)\bigr) \,\dd t + \sqrt{2a^0\bigl(X_t,u(t)\bigr)} \,\dd W_t, \\
   \mathcal{L}(X_t) &= u(t,x)\,\dd x, \ t>0.
    \end{aligned}
    \right.
    \end{equation}
Here $a^0 = (a^0_{ij})_{1\leq i,j \leq d}$ takes values in the space of symmetric non-negative definite matrices, and $b^0 = (b^0_1,\dots,b^0_d)$. 
In \eqref{gen:lin-fpe}, $t \mapsto v(t)$ is a curve of probability densities which is fixed a priori in the second arguments of $a^0$ and $b^0$.
%Their notion of \emph{probability solution} and \emph{probabilistically weak) solution}, respectively, is recalled in the appendix for the sake of the reader. 
Note that \eqref{eq:NL-FPE}, \eqref{eq:lin-FPE} and \eqref{eq:DDSDE} are special cases of \eqref{gen-nl-fpe}, \eqref{gen:lin-fpe} and \eqref{gen-DDSDE}, respectively. 
For the standard definition of solutions to the above equations, see for instance \cite[Definition 3.1]{RR25}.

The following general result from \cite[Theorem 4.6]{ABR26} (see also \cite[Theorem 3.8]{RR25}) %\blue{[SO: Are these theorem numbers correct? Since this an essential ingredient, we may say a few words about its availability.]} \red{M: yes, all correct. At this point I think no need to say more, since we just re-state the result from [1] as Prop.4.3; then we verify its applicability later on} 
is one key ingredient for the proof of our main uniqueness result in Theorem~\ref{thm:main-result} below. 
For $s\geq 0$, we write $u^z_{s+}$ for the map $u^z_{s+}(t,x) := u^z(s+t,x)$, $(t,x) \in (0,\infty)\times \R^d$. 
Note that since \eqref{gen:lin-fpe} is a linear equation, its set of solutions with a prescribed initial datum is convex.

\begin{proposition}\label{prop:NL-MC-construction}
    For each $z \in \R^d$, let $u^z$ be a probability solution to the nonlinear Fokker--Planck equation \eqref{gen-nl-fpe} with initial datum $z$, such that
    \begin{enumerate}[(i)]
        \item $(t,z) \mapsto u^z(t,x) \,\dd x$ is injective from $[0,\infty)\times \R^d$ to  $\Pscr$ \textnormal{(}here we write $u^z(0,x) \,\dd x := \delta_z)$;

        \item for each $z \in \R^d$ and $s \in [0,\infty)\setminus \{0\}$, $u^z_{s+}$ is an extreme point in the convex set of probability solutions to \eqref{gen:lin-fpe} with $u^z(s+t)$ replacing $v(t)$ and with initial datum $u^z(s,x) \,\dd x$.
    \end{enumerate}
    Then, for each $z \in \R^d$, there exists a unique weak solution $X^z$ to the McKean--Vlasov SDE \eqref{gen-DDSDE} with the prescribed one-dimensional time marginals $\mathcal{L}(X^z_t) = u^z(t,x) \,\dd x$, $t > 0$, and $X^z_0 = z$. 
    The family of solution path laws $\{P^z\}_{z\in \R^d}$, $P^z = \mathcal{L}(X^z)$, is a nonlinear Markov core for $\{u^z(t)\}_{t>0, z\in \R^d}$ and is uniquely determined by $\{u^z\}_{z \in \R^d}$ and \eqref{gen-DDSDE}.
\end{proposition}

For later use, we also recall a crucial lemma from \cite[Lemma 3.5]{RR25}, which characterizes the extremality condition in Proposition \ref{prop:NL-MC-construction}(ii) by a restricted uniqueness condition for the associated linearized Fokker--Planck equation.

\begin{lemma}\label{lem:equiv-extrem-uniqu}
  Let $(s,z) \in (0,\infty) \times \R^d$. 
  In the setting of the previous proposition, $u^z_{s+}$ is an extreme point in the set of probability solutions to \eqref{gen:lin-fpe} with initial datum $u^z(s,x) \,\dd x$ if and only if $u^z_{s+}$ is a unique probability solution to \eqref{gen:lin-fpe} with initial datum $u^z(s,x) \,\dd x$ in the set 
    \begin{equation}\label{eq:set1}
        \bigl\{[t\mapsto v(t,x) \,\dd x] \in C_w([0,\infty);\Pscr)\,|\, \exists C \geq 1, \, v(t,x) \leq C u^z_{s+}(t,x) \,\dd x \dd t\text{-a.e.} \bigr\},
    \end{equation}
    where $C_w([0,\infty);\Pscr)$ denotes the set of weakly continuous paths $t\mapsto \mu_t \in \Pscr$ on $[0,\infty)$.   
\end{lemma}

\subsection{Brownian motion in Minkowski normed spaces and nonlinear Markov property}

The following theorem is our main result of this paper. 

\begin{theorem}[Uniqueness]
\label{thm:main-result}
    Let $d\geq 2$.
    \begin{enumerate}[(i)]
        \item  For each $z\in \R^d$, the solution $X^z$ from Theorem \ref{thm:construction} is the unique weak solution to \eqref{eq:DDSDE} with the prescribed one-dimensional time marginals $\mathcal{L}(X^z_t) = f^z_Q(t,x)\,\dd x$, $t>0$, and $X^z_0 = z$;
        \item The family of path laws $\{P^z\}_{z\in \R^d} \subseteq \mathcal{P}(\mathcal{C})$, $P^z := \mathcal{L}(X^z)$, is a nonlinear Markov core for $\{f^z_Q(t)\}_{t>0, z\in \R^d}$.
    \end{enumerate}
\end{theorem}

Via this result, we construct a probabilistic counterpart of the Finsler heat kernel as the nonlinear Markov process consisting of the family of unique solutions to the McKean--Vlasov SDE \eqref{eq:DDSDE} with Finsler heat kernel one-dimensional time marginal densities.

\begin{remark}[Comparison with Euclidean case]
In the special case $Q(x) = Q_2(x) := |x|^2/2$, where $|\cdot|$ denotes the standard Euclidean norm on $\R^d$, this construction yields standard Brownian motion. 
Indeed, in this case $f^z_{Q_2}$ is the standard heat kernel
\[
f^z_{Q_2}(t,x) = \frac{1}{(2\pi t)^{d/2}} \exp\biggl( - \frac{|z-x|^2}{2t}\biggr),
\]
and the Finsler Laplacian becomes the standard Laplacian, i.e., \eqref{eq:FHE} becomes the standard heat equation.
Likewise, \eqref{eq:DDSDE} reduces to $\dd X_t = \dd W_t$.
Clearly, by imposing the initial condition $X_0 = z$, the unique solutions are simply $X^z_t := W_t +z$ for a standard Brownian motion $W$. 
Hence, in this case the nonlinear Markov core constructed in Theorem \ref{thm:main-result} is actually a classical Markov process, consisting of the shifted Wiener measures $\mathbb{W}^z = \mathcal{L}(W+z)$. 
\end{remark}

Having this analogy to the standard Brownian motion in the Euclidean case, we propose the following definition.
\begin{definition}[Brownian motion in Minkowski normed spaces]
\label{def:FBm}
 Let $d\geq 2$. For $z\in \R^d$, we call the solution $X^z$ to \eqref{eq:DDSDE} from Theorem \ref{thm:main-result}(i) (defined on any stochastic basis supporting a standard Brownian motion) \emph{Brownian motion starting at $z$ in the Minkowski normed space $(\R^d,\|\cdot\|)$}.
\end{definition}

\section{Proof of Theorem \ref{thm:main-result}}\label{sect:proof}

All assertions of Theorem \ref{thm:main-result} will follow by applying Proposition \ref{prop:NL-MC-construction} to $u^z = f^z_Q$, \eqref{eq:NL-FPE} and \eqref{eq:DDSDE}. 
Thus, below we verify its assumptions for this case. 
Injectivity of $(t,z) \mapsto f^z_Q(t,x) \,\dd x$ from $[0,\infty)\times \R^d$ to $\Pscr$ is obvious (again, with slight abuse of notation, we write $f^z_Q(0,x) \,\dd x = \delta_z)$. 
Hence, taking Lemma \ref{lem:equiv-extrem-uniqu} into account, the following lemma verifies all assumptions of Proposition \ref{prop:NL-MC-construction}, which in turn completes the proof of Theorem \ref{thm:main-result}. 

\begin{lemma}\label{lem:crucial-lemma}
    Let $d\geq 2, s> 0$ and $z\in \R^d$. 
     Then the linear equation \eqref{eq:lin-FPE} has a probability solution with initial datum $\zeta = f^z_Q(s,x) \,\dd x$ in the sense of Definition \ref{def:lin-FPE}, which is the unique solution in
    \begin{align}
     &\mathcal{S}_{s,z}:= \notag\\
     &\,\bigl\{ [t\mapsto u(t,x)\,\dd x] \in C_w([0,\infty);\Pscr)\,|\, \exists C\geq 1,\, u(t,x) \leq Cf^z_Q(s+t,x)\, \dd x\dd t\text{-a.e.} \bigr\}.
    \end{align}
\end{lemma}

By Lemma \ref{lem:f-solves-linfpe}, this solution is identified as~$u(t,x) = f^z_Q(s+t,x)$. 

\begin{remark}
    As the proof below shows (see in particular Proposition \ref{prop:pathwise-unique}), we in fact prove a strictly stronger statement, namely uniqueness in the class of \emph{all} probability solutions with initial datum $\zeta = f^z_Q(s,x) \,\dd x$.
\end{remark}

Before proving the previous lemma, let us point out that while its assertion concerns uniqueness for a \emph{linear} PDE, a topic with a long and well-developed history,  its proof is not standard in the present setting because 
\begin{itemize}
    \item  we claim uniqueness in the potentially large class of \emph{distributional} solutions -- and not (as would be much simpler), for instance, in the class of \emph{weak} solutions which have weak first order spatial derivatives (cf.\ \cite{Ohta2021,OhtaSturm2009}). 
    In fact, while $f^z_Q$ is regular, a priori this need not be true for elements in~$\mathcal{S}_{s,z}$.
    \item our coefficients have critical/supercritical singularities as seen in~Remark~\ref{r:BS}.
\end{itemize}

\subsection{Proof of Lemma \ref{lem:crucial-lemma}}

Lemma \ref{lem:crucial-lemma} will follow as a consequence of the following strictly stronger statement. 

\begin{proposition}[Pathwise uniqueness]
\label{prop:pathwise-unique}
Let $d\geq 2$, $z\in \R^d$ and $s>0$. 
Probabilistically weak solutions to 
\begin{equation}\label{SDE-linear}
    \dd X_t = \frac 1 2 \divv \Bigl(D^2 Q^*\bigl(\nabla Q(z-X_t)\bigr)\Bigr) \,\dd t + \sqrt{D^2Q^*\bigl(\nabla Q(z-X_t)\bigr)} \,\dd W_t
\end{equation}
$($that is, to the first equation in \eqref{eq:f-DDSDE}$)$ with initial datum $\mathcal{L}(X_0) = f^z_Q(s,x)\,\dd x$ are \emph{pathwise unique}, i.e., $X_t = \widetilde{X}_t$ for all $t\geq 0$ a.s.\ for any two probabilistically weak solutions $X$ and $\widetilde{X}$ to \eqref{SDE-linear} with initial datum $f^z_Q(s,x) \,\dd x$ on the same stochastic basis with the same standard Brownian motion $W$ such that $X_0 = \widetilde{X}_0$ a.s.
\end{proposition}

From this result, one obtains Lemma \ref{lem:crucial-lemma} as follows.

\begin{proof}[Proof of Lemma \ref{lem:crucial-lemma}.]
Suppose that $u$ and $\tilde{u}$ are two probability solutions to \eqref{eq:lin-FPE} with initial condition  $f^z_Q(s,x)\,\dd x$. 
Then, by the superposition principle (Proposition \ref{prop:SP}), there exist probabilistically weak solutions $X$, $\widetilde{X}$ to \eqref{SDE-linear} such that
\[
\mathcal{L}(X_t) = u(t,x) \,\dd x,\quad \mathcal{L}(\widetilde{X}_t) = \tilde{u}(t,x)\,\dd x,\quad \forall t \geq 0.
\]
Since Proposition \ref{prop:pathwise-unique} in particular implies weak uniqueness for \eqref{SDE-linear} with initial condition $f^z_Q(s,x)\,\dd x$, we obtain $\mathcal{L}(X_t) = \mathcal{L}(\widetilde{X}_t)$ for all $t\geq 0$, hence $u(t,x) = \tilde{u}(t,x)$ $\dd x$-a.e.\ for all $t \ge 0$.
\end{proof}

\subsection{Proof of Proposition \ref{prop:pathwise-unique}}
%&It applies to \emph{every} weak solution of the SDE \eqref{SDE-linear} with initial law $f_Q^z(s,\cdot)\dd x$, not only to the solutions coming from the PDE.

We first prove Proposition \ref{lem:no-hit-sde} below, which reveals a new universality property of Brownian motion in Minkowski normed spaces: the radial part always becomes a $d$-dimensional Bessel process independently of the choice of  the Minkowski structure~$Q$ under consideration. This universality plays a key role in proving pathwise uniqueness by establishing the inaccessibility of singular sets. It also appears to be of independent interest. We recall the definition and some facts of Bessel processes according to \cite{revuz.yor:99:continuous}.

\paragraph{Bessel processes.}
\begin{definition}[\cite{revuz.yor:99:continuous}, Definition 1.9, p.445]
Let $d\geq 2$. A nonnegative stochastic process $(Y_t)_{t\geq 0}$ on a stochastic basis is called \emph{$d$-dimensional Bessel process}, if $(Y^2_t)_{t\geq 0}$ solves the squared (one-dimensional) Bessel SDE
\begin{equation}\label{eq:sBSDE}
    \dd X_t = d \,\dd t + 2\sqrt{X_t}\,\dd\beta_t, \quad t \geq 0,
\end{equation}
where $\beta$ is a one-dimensional standard Brownian motion.
\end{definition}

\begin{remark}\label{rem:Bessel}
    \begin{enumerate}
        \item [(i)] The SDE \eqref{eq:sBSDE} has pathwise unique solutions for each initial condition $x\geq 0$ and $d\geq 2$ (see, for instance, \cite[p.439]{revuz.yor:99:continuous}), consequently also for each initial distribution $\mu_0$ concentrated on $[0,\infty)$.
        \item[(ii)] It is well-known that if $W$ is a $d$-dimensional standard Brownian motion and $\Sigma_0$ is a random variable independent from $W$, then $X_t := |\Sigma_0+W_t|^2$ is a weak solution to \eqref{eq:sBSDE}. In particular, point polarity of Brownian motion in $d\geq 2$ entails that solutions to \eqref{eq:sBSDE} are strictly positive a.s. Thus, any $d$-dimensional Bessel process $Y$ satisfies $\mathbb{P}(Y_t >0,\,\forall t >0) = 1$ (\cite[p.442]{revuz.yor:99:continuous}).
        \item[(iii)] The above implies that the path law of a $d$-dimensional Bessel process with initial condition $\mu_0$ equals the path law of $|\Sigma_0  + W|$, where $W$ is a $d$-dimensional standard Brownian motion and $\Sigma_0$ is independent of $W$ such that the distribution of $|\Sigma_0|$ is $\mu_0$.
    \end{enumerate}
\end{remark}

\paragraph{A universality result for Brownian motion in Minkowski normed spaces.}

%that the radial process is {\it always} $d$-dimensional Bessel process independently of the choice of $Q$.

\begin{proposition}\label{lem:no-hit-sde}
Let $s> 0$ and $X$ be a weak solution of \eqref{SDE-linear} with initial datum $f_Q^z(s,x)\,\dd x$, defined on a stochastic basis $(\Omega, \Fscr, (\Fscr_t)_{t\geq 0}, \mathbb{P})$ with a Brownian motion $W$.
Then, the radial process $\|z-X_t\|$ is a $d$-dimensional Bessel process.

In particular, for all $T>0$, we have
\begin{equation}\label{eq:no-hit}
\Pbb\bigl(X_t\ne z\text{ for all }t\in[0,T]\bigr)=1.
\end{equation}
\end{proposition}

Before giving the proof of this proposition and, based thereon, the proof of Proposition \ref{prop:pathwise-unique} (which in turn concludes the proof of Lemma \ref{lem:crucial-lemma} and thereby the proof of Theorem \ref{thm:main-result}), we present a notable universality property of Brownian motion in Minkowski normed spaces which readily follows from the previous proposition.

\begin{theorem}\label{thm:universality}
    Let $d\geq 2$, $z\in \R^d$,  and $X^z$ a Brownian motion on $(\R^d, \|\cdot\|)$ with initial condition $z$ in the sense of Definition \ref{def:FBm}. Let $W$ denote a standard Brownian motion. Then the path law of $(\|z-X^z_t\|)_{t\geq 0}$ is the one of the standard $d$-dimensional Bessel process $(|W_t|)_{t\geq 0}$.
\end{theorem}

\begin{proof}
    Proposition \ref{lem:no-hit-sde} implies that, for all $s>0$,
    $(\|z-X^{z}_t\|)_{t\geq s}$ is a $d$-dimensional Bessel process, i.e., its path law equals the path law of $(|\Sigma_s+ W_t|)_{t\geq 0}$, with $\Sigma_s$ independent of $W$ and such that $|\Sigma_s|$ has the same distribution as $\|\Lambda_s\|$, where $\Lambda_s$ has density $f^0_Q(s,x)$.
    Using the definition of $f^0_Q(s,x)$, a straightforward calculation entails that the distribution of $\|\Lambda_s\|$ is the same as the distribution of the modulus of a centered $d$-dimensional Gaussian with covariance $s \text{Id}$.
    Consequently, the path law of $(\|z-X^z_t\|)_{t\geq s}$ coincides with the one of $(|W_{t}|)_{t\geq s}$ for each $s>0$. Since also the distribution of $\|z-X^z_0\|$ and $|W_0|$ are equal (to $\delta_0$), it follows that (the family of finite-dimensional time marginals, and hence)
    the path law of $(\|z-X^z_t\|)_{t\geq 0}$ is the same as that of $(|W_t|)_{t\geq 0}$.
\end{proof}

\paragraph{Proofs of Propositions \ref{lem:no-hit-sde} and \ref{prop:pathwise-unique}.}

Now we turn to the proof of Proposition \ref{lem:no-hit-sde}.

\begin{proof}[Proof of Proposition \ref{lem:no-hit-sde}]
We set 
\[
Y_t:=2Q(z-X_t),\quad t\geq 0.
\]
%and shall prove $Y_t$ is the $d$-dimensional squared Bessel process, \red{which implies that $\sqrt{Y_t}=\|z-X_t\|$ is a $d$-dimensional Bessel process.}
%for all $t \in [0,T]$ a.s., which completes the proof.
Since $\mathcal{L}(X_0) = f_Q^z(s,x)\,\dd x$, one has $\Pbb(X_0=z)=0$, hence $Y_0>0$ a.s.
Fix $0<\varepsilon <\infty$ and consider the $(\Fscr_t)$-stopping time
\[
\tau_{\varepsilon}:=
\inf\{t\ge 0 \,|\, Y_t\le \varepsilon\}.
\]
% Choose $\chi_{\varepsilon}\in C_c^\infty(\R^d\setminus\{z\})$ such that
% \begin{align*}
% &0\le \chi_{\varepsilon,R}\le 1,
% \qquad
% \chi_{\varepsilon,R}\equiv 1\ \text{on } K_{\varepsilon,R}:=\{x\in\R^d \,|\, \varepsilon\le 2Q(z-x) \le R\},
% \\
% %\qquad
% &\operatorname{supp}\chi_{\varepsilon,R}\subset\{x\in\R^d \,|\, \varepsilon/2<2Q(z-x)<2R\},
% \end{align*}
%and 
% Define $q_{\varepsilon}(x):=\chi_{\varepsilon,R}(x)Q(z-x)\in C_c^3(\R^d)$.
% As $Q(z-\cdot) \in C^2(\R^d \setminus \{z\})$, 
% applying It\^o's formula up to time $\tau_{\varepsilon}$, we may compute $q_{\varepsilon,R}(X_t)$ and $q_{\varepsilon,R}(X_t)^2$ on the set $K_{\varepsilon,R}$ on which $q_{\varepsilon,R}=Q(z-\cdot)$ and all derivatives of $q_{\varepsilon,R}$ agree with those of $Q(z-\cdot)$. 
Observe that the linearized PDE \eqref{eq:lin-FPE}, which corresponds to the SDE \eqref{SDE-linear},  can be written as
\[
\partial_t u= \frac{1}{2}Lu, \qquad
Lu:=\divv\bigl( A(z-\cdot)\nabla u \bigr),
\]
for $A$ given in Lemma~\ref{lem:structure} (for sufficiently regular $u$). 
Indeed, it follows from Lemma \ref{lem:nice-form} that \eqref{eq:lin-FPE} becomes
\begin{align*}
\partial_t u
&= \partial_{ij} \biggl( \frac{1}{2} A_{ij}(z-\cdot) u \biggr) +\partial_i \biggl( \frac{1}{2} \partial_j A_{ij}(z-\cdot) u \biggr) \\
&= \partial_i \biggl( \frac{1}{2} A_{ij}(z-\cdot) \partial_j u \biggr).
\end{align*}

To apply It\^o's lemma, we compute $LQ(z-\cdot)$ and $L(Q(z-\cdot)^2)$.
First, by \eqref{eq:AgradQ} we have
\begin{equation}\label{e:LQD}
LQ(z-\cdot) =-\diver\bigl(A(z-\cdot)\nabla Q(z-\cdot)\bigr) =-\diver(z-\cdot)=d.
\end{equation}
Next,
\begin{align}
L\bigl(Q(z-\cdot)^2\bigr)-2Q(z-\cdot)LQ(z-\cdot)
&=2\nabla Q(z-\cdot)\cdot \bigl( A(z-\cdot)\nabla Q(z-\cdot) \bigr) \notag\\
&=2\nabla Q(z-\cdot) \cdot (z-\cdot) \notag\\
&=4Q(z-\cdot), \label{e:QFo}
\end{align}
where we used \eqref{eq:AgradQ} again as well as Lemma \ref{lem:structure}(ii).

\begin{claim}\label{c:Ito}
The process $Y_{t}$ satisfies \eqref{eq:sBSDE}
on $[0,\tau_\varepsilon]$, where $\beta_t$ is a $1$-dimensional $(\Fscr_t)$-Brownian motion on $(\Omega, \Fscr, (\Fscr_t)_{t\geq 0}, \mathbb{P})$ constructed below.
% \[
% M_t^{\varepsilon}:=Y_{t\wedge\tau_{\varepsilon}}-Y_0-d(t\wedge\tau_{\varepsilon})
% \]
% is a continuous local martingale with quadratic variation
% \[
% [M^{\varepsilon,R}]_t=8\purple{(4)}\int_0^{t\wedge\tau_{\varepsilon}} Y_r\,\dd r.
% \]
\end{claim}

\begin{proof}
As $Q(z-\cdot) \in C^2(\R^d \setminus \{z\})$, 
applying It\^o's formula to $2Q(z-X_t)$ with the SDE~\eqref{SDE-linear} up to the stopping time~$\tau_{\varepsilon}$ and using the expression $\partial_t u= \frac{1}{2}Lu$ of the linearized PDE~\eqref{eq:lin-FPE}, we obtain 
%Since
%$f_{\varepsilon,R}:=2q_{\varepsilon,R}\in C_c^\infty(\R^d)$,
%It\^o's formula applied with the expression $\partial_t u= \frac{1}{2}Lu$ of the linearized PDE~\eqref{eq:lin-FPE} shows that
\begin{align} \label{e:STY}
&Y_{t \wedge \tau_\varepsilon} -Y_0  \notag
\\
&= \int_0^{t\wedge \tau_\varepsilon}  LQ(z-X_r)\,\dd r + 2\int_0^{t\wedge \tau_\varepsilon}\bigl(\nabla Q(z-X_r)\bigr)^T [D^2Q(z-X_r)]^{-1/2}\,\dd W_t. 
\end{align}
% is a continuous local martingale, hence so is its stopped version at $\tau_{\varepsilon,R}$.
% By~\eqref{e:LQD}, we observe
% \[
%  LQ(z-X_r)=d.
% \]
Set 
\[
q_r := \frac{(\nabla Q(z-X_r))^T [D^2Q(z-X_r)]^{-1/2}}{\sqrt{2Q(z-X_r)}}
\]
and define $q_r:=e$ for $r\ge \tau_\varepsilon$, where $e$ is any fixed unit vector. 
Recall~$\nabla Q(x) \cdot x =2Q(x)$ in~Lemma~\ref{lem:structure} and $D^2 Q(x)x = \nabla Q(x)$ in its proof.
Thus, 
\[
[D^2Q(x)]^{-1} \nabla Q(x)=x.
\]
This implies 
\begin{align*}
|q_r|^2&=\frac{(\nabla Q(z-X_r))^T [D^2Q(z-X_r)]^{-1}\nabla Q(z-X_r)}{2Q(z-X_r)}
\\
&=\frac{2Q(z-X_r)}{2Q(z-X_r)} = 1.
\end{align*}
Therefore,
\[
\beta_t = \int_0^t q_r \, \dd W_r
\]
is a Brownian motion by L\'evy's characterization, because the quadratic variation is $[\beta]_t = \int_0^t \,\dd r = t$. 
Combining these arguments, we obtain
\begin{align*}
&\int_0^{t \wedge \tau_\varepsilon} \sqrt{2Q(z-X_r)} \, \dd \beta_r \\
&= \int_0^{t \wedge \tau_\varepsilon} \sqrt{2Q(z-X_r)} q_r\, \dd W_r
\\
&=\int_0^{t \wedge \tau_\varepsilon} \sqrt{2Q(z-X_r)} \frac{(\nabla Q(z-X_r))^T [D^2Q(z-X_r)]^{-1/2}}{\sqrt{2Q(z-X_r)}}\, \dd W_r
\\
&=\int_0^{t \wedge \tau_\varepsilon}  \bigl(\nabla Q(z-X_r)\bigr)^T [D^2Q(z-X_r)]^{-1/2}\, \dd W_r.
\end{align*}
Hence, combined with~\eqref{e:LQD} and~\eqref{e:STY}, $Y$ satisfies \eqref{eq:sBSDE} on $[0,\tau_\varepsilon]$
as claimed.
% Thus, it follows that
% $M_t^{\varepsilon,R}
% =Y_{t\wedge\tau_{\varepsilon,R}}-Y_0-d(t\wedge\tau_{\varepsilon,R})$
% is a continuous local martingale.
% The standard relation between the generator and the quadratic variation (see, e.g., \cite[Proposition~4.7.11, eq~(4.79)]{Kolokoltsov2011}) yields
% \[
% [M^{\varepsilon,R}]_t
% =
% \purple{\frac{1}{2}}\int_0^{t\wedge\tau_{\varepsilon,R}}
% \bigl(L(f_{\varepsilon,R}^2)-2f_{\varepsilon,R}Lf_{\varepsilon,R}\bigr)(X_r) \,\dd r.
% \]
% Since on $\{r<\tau_{\varepsilon,R}\}$ one has $f_{\varepsilon,R}(X_r)=Y_r$, using \eqref{e:QFo}, we obtain
% \[
% L(f_{\varepsilon,R}^2)(X_r)-2f_{\varepsilon,R}(X_r)Lf_{\varepsilon,R}(X_r)=8\purple{(4)}Y_r.
% \]
% This completes the proof of Claim \ref{c:Ito}.
\end{proof}

We resume the proof of~Proposition~\ref{lem:no-hit-sde}.
Let $Z$ be the $d$-dimensional squared Bessel process on the same stochastic basis $(\Omega, \Fscr, (\Fscr_t)_{t\geq 0},\mathbb{P})$ and with the same Brownian motion $\beta$ as above with $Z_0=Y_0$ $\mathbb{P}$-a.s., i.e., the pathwise unique solution to 
\eqref{eq:sBSDE} with initial condition $Y_0$.
It suffices to show $Y = Z$.
We set the following notation:
\[
\sigma_{\varepsilon}:=\inf\{t\geq 0 \,|\, Z_t\le \varepsilon\}.
%\qquad
%\sigma^{R}:=\inf\{t\geq 0 \,|\, Z_t\ge R\},
\]
%\sout{Since~the SDE~\eqref{e:DBS} has a Lipschitz coefficient on $(\varepsilon, \infty)$, Claim~\ref{c:Ito} and the standard uniqueness result of SDEs shows that }
%\begin{align} \label{e:EZY}
%\sout{\mathcal L(Z_{\cdot \wedge \sigma_\varepsilon}) = \mathcal L(Y_{\cdot \wedge \tau_\varepsilon}).}
%\end{align}
Pathwise uniqueness for \eqref{eq:sBSDE} (up to $\tau_\varepsilon$) gives
\begin{equation}
    Z_{\cdot \wedge \tau_\varepsilon} = Y_{\cdot \wedge \tau_\varepsilon} \quad \text{a.s.}
\end{equation}
%hence also
%\begin{equation}
 %   Z_{\cdot \wedge \tau_\varepsilon \wedge \sigma_\varepsilon} = Y_{\cdot \wedge \tau_\varepsilon \wedge \sigma_\varepsilon},\quad a.s.,
%\end{equation}
Thus, $\mathbb{P}(\tau_\varepsilon < \sigma_\varepsilon) = 0$, since on $\{\tau_\varepsilon < \sigma_\varepsilon\}$, we have $Y_{\tau_\varepsilon} \leq \varepsilon$ and $Z_{\tau_\varepsilon} > \varepsilon$. 
Similarly, on $\{\sigma_\varepsilon < \tau_\varepsilon\}$, we have $Z_{\sigma_\varepsilon} \leq \varepsilon$ and $Y_{\sigma_\varepsilon} >\varepsilon$, hence combined 
\[
\mathbb{P}(\tau_\varepsilon = \sigma_\varepsilon) = 1,
\]
and, consequently, 
\[
Z_{\cdot \wedge \sigma_\varepsilon} = Y_{\cdot \wedge \tau_\varepsilon} \quad \text{a.s.}
\]
By Remark \ref{rem:Bessel}(ii), we have, for all $N \in \N$,
\[
\lim_{\varepsilon \to 0}\mathbb P[\sigma_\varepsilon <N] \to 0.
\]
As $\tau_\varepsilon$ is increasing as $\varepsilon \to 0$, we obtain for each $N \in \N$
\begin{align*}
\mathbb{P}[ \exists t \leq  N: Y_t = 0]
&= \mathbb P\bigl[ \bigcap_{\varepsilon > 0} \{\tau_\varepsilon <N\} \bigr]
= \lim_{\varepsilon \to 0} \mathbb P[\tau_\varepsilon <N] \\
&= \lim_{\varepsilon \to 0} \mathbb P[\sigma_\varepsilon <N]=0.
\end{align*}
Thus, $\mathbb{P}[\exists t\geq 0: Y_t = 0]= 0$, and so, $\mathbb{P}$-a.s., 
\[
\lim_{\varepsilon \to 0} Z_{\cdot \wedge \sigma_\varepsilon} = Z_{\cdot}, \qquad \lim_{\varepsilon \to 0}Y_{\cdot \wedge \sigma_\varepsilon} = Y_{\cdot}
\]
locally uniformly in $t\geq 0$. 
Hence, $Z_t = Y_t$, for all $t\geq 0$, $\mathbb{P}$-a.s.
\qedhere
\end{proof}

Now we are ready to prove Proposition \ref{prop:pathwise-unique}.

\begin{proof}[Proof of Proposition \ref{prop:pathwise-unique}.]
Let $X$ and $\widetilde X$ be two weak solutions of \eqref{SDE-linear} on the same stochastic basis, driven by the same Brownian motion such that $\mathcal{L}(X_0) = \mathcal{L}(\widetilde{X}_0) = f^z_Q(s,x) \,\dd x$ and $X_0 = \widetilde{X}_0$ a.s. 
By Proposition~\ref{lem:no-hit-sde}, 
\[
\Pbb\bigl(X_t\ne z\text{ and }\widetilde X_t\ne z\text{ for all }t\in[0,T]\bigr)=1.
\]
Fix $0<\varepsilon<R<\infty$ and define
\[
\tau_{\varepsilon,R} :=\inf\bigl\{t\ge 0 \,|\, Y_t\le \varepsilon,\, \widetilde Y_t\le \varepsilon,\, Y_t\ge R,\, \text{or}\ \widetilde Y_t\ge R\bigr\},
\]
where $Y_t:=2Q(z-X_t)$ and $\widetilde Y_t:=2Q(z-\widetilde X_t)$.
On the compact annulus
\[
K_{\varepsilon,R}:=\bigl\{x\in\R^d \,|\, \varepsilon\le 2Q(z-x) \le R\bigr\},
\]
the coefficients of the SDE~\eqref{SDE-linear} are smooth, hence Lipschitz continuous. 
Applying It\^o's formula to $|X_{t\wedge\tau_{\varepsilon,R}}-\widetilde X_{t\wedge\tau_{\varepsilon,R}}|^2$, taking expectation, and using this Lipschitz continuity gives
\[
\E\bigl|X_{t\wedge\tau_{\varepsilon,R}}-\widetilde X_{t\wedge\tau_{\varepsilon,R}}\bigr|^2
\le C_{\varepsilon,R}\int_0^t
\E\bigl|X_{r\wedge\tau_{\varepsilon,R}}-\widetilde X_{r\wedge\tau_{\varepsilon,R}}\bigr|^2 \,\dd r
\]
for some $C_{\varepsilon, R} >0$.
Since $X_0 = \widetilde{X}_0$ a.s., Gr\"onwall's lemma implies
\[
\E\bigl|X_{t\wedge\tau_{\varepsilon,R}}-\widetilde X_{t\wedge\tau_{\varepsilon,R}}\bigr|^2=0,
\qquad \forall 0\le t\le T,
\]
so that
\[
X_{t\wedge\tau_{\varepsilon,R}}=\widetilde X_{t\wedge\tau_{\varepsilon,R}}
\qquad\text{a.s. for every }t\in[0,T].
\]
Because both $Y$ and $\widetilde Y$ are strictly positive on $[0,T]$ and have continuous paths, we have $T \wedge \tau_{\varepsilon,R}\uparrow T$ a.s.\ as $\varepsilon\downarrow 0$ and $R\uparrow\infty$. 
This yields $X_t = \widetilde{X}_t$ a.s.\ for every $t\in [0,T)$, therefore, the path continuity of $X$ and $\widetilde{X}$ shows $X = \widetilde{X}$ on $[0,T]$ a.s. 
\end{proof}

\begin{remark}\label{rm:s>0}
    We stress that the above proof of Proposition~\ref{lem:no-hit-sde} crucially relies on $s>0$, more precisely, on the initial condition $f^z_Q(s,x)\,\dd x$ being absolutely continuous with respect to the Lebesgue measure. 
    This shows that, in the condition (ii) of Proposition \ref{prop:NL-MC-construction}, it is essential that one need not consider $s=0$. Notably, nevertheless, in Theorem \ref{thm:main-result} we obtain weak uniqueness of solutions to \eqref{eq:DDSDE} with Dirac initial conditions. 
\end{remark}

\section{Brownian motion 
in Minkowski normed spaces 
as probabilistically strong solutions}
\label{sect:strong}

Let us point out an additional consequence of Proposition \ref{prop:pathwise-unique}. 
Given $(s,z)\in (0,\infty) \times \R^d$, Proposition \ref{prop:pathwise-unique} yields pathwise uniqueness for \eqref{SDE-linear} with initial condition $f^z_Q(s,x) \,\dd x$. 
The existence of a probabilistically weak solution to \eqref{SDE-linear} with this initial condition follows from Theorem \ref{thm:construction} (indeed, consider $(X^z_{s+t})_{t\geq 0}$, where $X^z$ denotes the solution obtained in Theorem \ref{thm:construction}). 
Thus, the Yamada--Watanabe theorem implies that Brownian motion on a Minkowski normed space is adapted to the filtration of its driving Brownian motion from any strictly positive time $s>0$. 
More precisely, we have the following result.

\begin{proposition}\label{prop:strong-sol}
Let $d\geq 2$, $z\in \R^d$ and $s>0$. Solutions to \eqref{eq:DDSDE} with $u(t,x)$ replaced by $f^z_Q(s+t,x)$ and with initial datum $f^z_Q(s,x)\,\dd x$ are probabilistically strong and pathwise unique. 

In particular, if $(X^z_t)_{t\geq 0}$ is a Brownian motion in the Minkowski normed space under investigation with initial condition $z$ as in Definition \ref{def:FBm} on a stochastic basis, solving \eqref{eq:DDSDE} with a standard Brownian motion $W$, then $(X^z_{s+t})_{t\geq0}$ is adapted to the canonical filtration of $(W_{s+t}-W_s)_{t\geq0}$. 
\end{proposition}

\section{Examples of Minkowski normed spaces} \label{sec:E}

Finally, we give several examples of Minkowski norms that are not induced from inner products.
Recall that $|\cdot|$ denotes the standard Euclidean norm on $\R^d$.
We first discuss asymmetric norms arising from fundamental examples of irreversible Finsler metrics.

\begin{example}[Asymmetric norms] \ 
\label{ex:Randers}
\begin{enumerate}[(a)]
\item
\emph{Randers metrics} provide an important class of irreversible Finsler structures (see \cite{BCS,Ohta2021}).
In our setting, it is given as, for fixed $a \in \R^d \setminus \{0\}$ with $|a|<1$,
\[
\|x\| := |x| +a \cdot x.
\]
The unit sphere $\{x \in \R^d \,|\, \|x\|=1 \}$ is an ellipsoid whose center is not the origin, and $\|{-a}\| \neq \|a\|$.
In this case, one can explicitly calculate $A$ and $B$ as in Lemma \ref{lem:structure} as follows (see \cite[Lemma 6.14]{Ohta2021}): 
\begin{align*}
A_{ij}(x) &= \frac{|x|}{\|x\|} \delta_{ij} + \frac{a \cdot x +|a|^2 |x|}{\|x\|^3} x_i x_j - \frac{|x|}{\|x\|^2} (a_i x_j +a_j x_i), \\
B_i(x) &= \frac{d+1}{\|x\|^3} \bigl( (a \cdot x) x_i +|a|^2 |x| x_i -\|x\| |x| a_i \bigr),
\end{align*}
for $x \in \R^d \setminus \{0\}$. 
Thus, for $i$ with $a_i \neq 0$, we have $B_i(x) \neq 0$ when $x_i=0$.

\item
More generally, for $a \in \R^d \setminus \{0\}$ with $|a|<1$ and a positive $C^\infty$-function $\phi:(-1,1) \to (0,\infty)$ satisfying 
\[
\phi(s) -s\phi'(s) +(b^2-s^2)\phi''(s)>0
\]
for some $b \in (0,1)$ and all $s \in [-b,b]$, the associated \emph{$(\alpha,\beta)$-metric} is defined by
\[
\|x\| :=|x| \cdot \phi\biggl( \frac{a \cdot x}{|x|} \biggr) \,\ (x \neq 0), \qquad
\|0\| :=0
\]
(see, e.g., \cite{ShenShen}).
Note that choosing $\phi(s)=1+s$ recovers the Randers metric for $a$.
\end{enumerate}
\end{example}

Next, we consider symmetric norms.

\begin{example}[Symmetric norms] \ 
\begin{enumerate}[(a)]
\item
For $d\ge 2$ and $\varepsilon\in\mathbb R$, define
\[
\|x\|
:= \biggl( |x|^2 + \varepsilon\,\frac{x_1^4+\cdots+x_d^4}{|x|^2} \biggr)^{1/2}
\,\ (x\neq 0),
\qquad
\|0\|:=0.
\]
When $|\varepsilon|>0$ is sufficiently small, $Q_\varepsilon(x):=\|x\|^2/2$ satisfies our hypotheses \textup{(H1)--(H3)}.
Indeed, since $Q_\varepsilon(x)=O(|x|^2)$ and $\nabla Q_\varepsilon(x)=O(|x|)$ near $0$, $\nabla Q_\varepsilon(0)=0$ and $Q_\varepsilon\in C^1(\mathbb R^d)$.
Moreover, by $D^2Q_0=I$ and the continuity of $(\varepsilon,\theta)\mapsto D^2Q_\varepsilon(\theta)$ for $\theta \in S^{d-1}$, when $|\varepsilon|$ is sufficiently small, we find $\lambda,\Lambda>0$ such that
    \[
    \lambda I\le D^2Q_\varepsilon(\theta)\le \Lambda I
    \qquad \forall \theta\in S^{d-1}.
    \]
To see that $\|\cdot\|$ does not come from an inner product, 
let $e_1,e_2$ be the first two basis vectors.
Then 
$\|e_1\|^2=\|e_2\|^2=1+\varepsilon$
and $\|e_1+e_2\|^2=\|e_1-e_2\|^2=2+\varepsilon$.
Therefore, the parallelogram identity fails (and hence $\|\cdot\|$ is not induced by an inner product) unless $\varepsilon=0$:
\[
2\|e_1\|^2 +2\|e_2\|^2
-\|e_1+e_2\|^2 -\|e_1-e_2\|^2
=2\varepsilon
\neq 0.
\]

\item
A norm $\|\cdot\|$ is determined by its unit sphere $S:=\{x \in \R^d \,|\, \|x\|=1 \}$;
then conditions (H1)--(H3) are met when $S$ is $C^3$ and positively curved.
For example, in the polar coordinates $(r,\theta)$ of $\R^2$, given $\varepsilon \in (0,1/3)$, set
\[
S:=\{ (r,\theta) \,|\, r=\rho(\theta) \}, \qquad
\rho(\theta) :=\sqrt{1+\varepsilon\cos 2\theta}.
\]
Then, the associated norm is
\[
\|(r,\theta)\| =\frac{r}{\rho(\theta)} \,\ (r>0), \qquad
\|0\| =0.
\]
To show that this norm is uniformly convex, observe that
for a function of the form
$Q(r,\theta)=r^2 q(\theta)$ with $q>0$,
the Hessian in the orthonormal polar frame $(e_r,e_\theta)$ is
\[
D^2Q=
\begin{pmatrix}
2q & rq'\\[1mm]
rq' & r^2(q''+2q)
\end{pmatrix}.
\]
Hence, $D^2Q$ is positive definite on $\R^2 \setminus \{0\}$ if and only if
\[
2q(q''+2q)-(q')^2>0.
\]
For $q(\theta)=1/(1+\varepsilon\cos 2\theta)$, a direct computation gives
\[
2q(q''+2q)-(q')^2
=\frac{4(1+3\varepsilon^2 +4\varepsilon\cos 2\theta)}{(1+\varepsilon\cos 2\theta)^4}
\ge \frac{4(1+3\varepsilon^2 -4\varepsilon)}{(1+\varepsilon\cos 2\theta)^4},
\]
which is positive provided $\varepsilon \in (0,1/3)$.
The failure of the parallelogram identity can be seen in the same way as (a):
\begin{align*}
2\|e_1\|^2 +2\|e_2\|^2
-\|e_1+e_2\|^2 -\|e_1-e_2\|^2
&=\frac{2}{1+\varepsilon} +\frac{2}{1-\varepsilon} -4 \\
&=\frac{4\varepsilon^2}{1-\varepsilon^2}
\neq 0.
\end{align*}

\item
One can generalize the construction in (b) to
\[
\|x\|=|x|\biggl(1+\varepsilon\psi\biggl(\frac{x}{|x|}\biggr)\biggr) \,\ (x \neq 0),
\qquad \|0\|=0,
\]
where $\psi\in C^\infty(S^{d-1})$ is even and nonconstant. 
When $|\varepsilon|$ is sufficiently small, this norm satisfies \textup{(H1)--(H3)}.
\end{enumerate}
\end{example}

%%%%

%%% for ARMA submission %%%
\paragraph{Data availability statement}
This manuscript has no associated data.

\paragraph{Conflict of interest}
We declare that there is no conflict of interest in this work.

\bibliographystyle{plain}
\bibliography{refs}

\end{document}